\documentclass[11pt, twoside, letterpaper]{amsart}
\usepackage{amsmath, hyperref, color}
\allowdisplaybreaks

\newtheorem{thm}{Theorem}[section]
\newtheorem{cor}[thm]{Corollary}
\newtheorem{lem}[thm]{Lemma}
\newtheorem{prop}[thm]{Proposition}

\theoremstyle{definition}
\newtheorem{defn}[thm]{Definition}

\newtheorem{ass}[thm]{Assumption}

\theoremstyle{remark}
\newtheorem{rem}[thm]{Remark}

\numberwithin{equation}{section}

\newcommand{\Real}{\mathbb R}
\newcommand{\Natural}{\mathbb N}
\newcommand{\R}{\mathbb{R}}
\newcommand{\N}{\mathbb{N}}

\newcommand{\A}{\mathcal{A}}
\newcommand{\such}{\ | \ }
\newcommand{\nin}{n \in \Natural}
\newcommand{\prob}{\mathbb{P}}
\newcommand{\M}{\mathcal{M}}

\newcommand{\cE}{\mathcal{B}_E}
\newcommand{\FF}{\bF}

\newcommand{\PP}{\mathbb{P}}

\newcommand{\expec}{\mathbb{E}}

\newcommand{\basisp}{(\Omega, \, \bF, \, \prob)}

\newcommand{\basisgp}{(\Omega, \, \bG, \, \prob)}

\newcommand{\F}{\mathcal{F}}
\newcommand{\G}{\mathcal{G}}

\newcommand{\ud}{\mathrm d}

\newcommand{\cB}{\mathcal{B}}
\newcommand{\cO}{\mathcal{O}}
\newcommand{\cP}{\mathcal{P}}
\newcommand{\Stilde}{\widetilde{S}}
\newcommand{\Mloc}{\M_{\text{loc}}}
\newcommand{\mMloc}{\M_{\text{\emph{loc}}}}

\newcommand{\pare}[1]{\left(#1\right)}

\newcommand{\dbra}[1]{[\kern-0.15em[ #1 ]\kern-0.15em]}
\newcommand{\dbraco}[1]{[\kern-0.15em[ #1 [\kern-0.15em[}
\newcommand{\dbraoc}[1]{]\kern-0.15em] #1 ]\kern-0.15em]}
\newcommand{\dbraoo}[1]{]\kern-0.15em] #1 [\kern-0.15em[}

\newcommand{\Shat}{\widehat{S}}

\newcommand{\oF}{\widehat{\F}}

\newcommand{\bF}{\mathbf{F}}

\newcommand{\bG}{\mathbf{G}}
\newcommand{\GG}{\mathbf{G}}
\newcommand{\ind}{\mathbf{1}}

\newcommand{\limn}{\lim_{n \to \infty}}

\newcommand{\dfn}{\, := \,}

\newcommand{\be}{\begin{equation}}
\newcommand{\ee}{\end{equation}}
\newcommand{\ba}{\begin{aligned}}
\newcommand{\ea}{\end{aligned}}

\newcommand{\oFF}{\widehat{\FF}}
\newcommand{\oPP}{\widehat{\PP}}
\newcommand{\oA}{\widehat{\A}}
\newcommand{\oE}{\widehat{\expec}}
\newcommand{\ptilde}{\widetilde{\prob}}

\begin{document}

\title[The strong predictable representation property in initially enlarged filtrations]{The strong predictable representation property \\ in initially enlarged filtrations \\ under the density hypothesis}%

\author[C. Fontana]{Claudio Fontana}%
\address{Claudio Fontana, Laboratoire de Probabilit\'es et Mod\`eles Al\'eatoires, Paris Diderot University (France)}%
\email{fontana@math.univ-paris-diderot.fr}%

\thanks{The author is grateful to Huy N. Chau, Andrea Cosso and the participants to the London-Paris Bachelier Workshop on Mathematical Finance 2015 for valuable discussions on the subject of this paper as well as to the associate editor and the referees for their detailed and incisive comments that helped to improve the paper.}%

\date{\today}%


\maketitle

\begin{abstract}
We study the strong predictable representation property in filtrations initially enlarged with a random variable $L$. We prove that the strong predictable representation property can always be transferred to the enlarged filtration as long as the classical density hypothesis of \cite{Jac85} holds. This generalizes the existing martingale representation results and does not rely on the equivalence between the conditional and the unconditional laws of $L$. 
Depending on the behavior of the density process at zero, different forms of martingale representation are established.
The results are illustrated in the context of hedging contingent claims under insider information.
\end{abstract}

\section{Introduction}

The theory of enlargement of filtrations aims at understanding the behavior of (semi-)martingales with respect to the introduction of additional information. In particular, an initial filtration enlargement corresponds to the introduction of the  information generated by some random variable $L$ to the initial $\sigma$-field $\F_0$ of a reference filtration $\FF=(\F_t)_{t\geq0}$, thus giving rise to the initially enlarged filtration $\GG=(\G_t)_{t\geq0}$. 
While the martingale property of a process is typically lost when passing from the reference filtration $\FF$ to the enlarged filtration $\GG$, the semimartingale property can be preserved under rather natural assumptions on the random variable $L$. In particular, in the seminal paper \cite{Jac85} it has been shown that every $\FF$-semimartingale is a $\GG$-semimartingale if the $\F_t$-conditional law of $L$ is absolutely continuous with respect to its unconditional law, for all $t\in\Real_+$. This condition (also called \emph{Jacod's density hypothesis}) has a prominent role in the theory of  enlargement of filtrations 
and has been widely employed in financial mathematics, notably in relation to the modeling of insider information (see e.g. \cite{GP98,AIS98,Amen,MR1831271,Bau03,Campi05,EL05,GVV06,ACJ15,AFK}).

In this paper we address the following fundamental question: suppose that every $\FF$-local martingale can be written as a stochastic integral of a given $\FF$-local martingale $S=(S_t)_{t\geq0}$, does there exist a $\GG$-local martingale $S^{\GG}=(S^{\GG}_t)_{t\geq0}$ (and, if yes, how is it related to $S$) such that every $\GG$-local martingale can be written as a stochastic integral of $S^{\GG}$? In other words, is it possible to transfer the strong predictable representation property from the original filtration $\FF$ to the initially enlarged filtration $\GG$? Assuming the validity of Jacod's density hypothesis, we shall give an answer to this question in full generality.

Since martingale representation results play a fundamental role in mathematical finance, stochastic filtering and backward stochastic differential equations, the above question has already been studied in several papers. In particular, martingale representation theorems in initially enlarged filtrations have been established first by \cite{GP98} in a Brownian setting and then by \cite{Amen,ABS,CJZ} in more general settings. However, to the best of our knowledge, the existing martingale representation results always assume a stronger version of Jacod's density hypothesis, namely the \emph{equivalence} between the $\F_t$-conditional law and the unconditional law of $L$, for all $t\in\Real_+$. In contrast, in this paper we shall only assume an \emph{absolute continuity} relation, as in the original paper \cite{Jac85}. This is a seemingly slight generalization of the existing literature, but on the contrary it requires a different approach to the martingale representation property. Moreover, it allows to study several interesting examples which are not covered by the existing results and are of special importance for the modeling of insider information.

If Jacod's density hypothesis holds as an equivalence between the conditional and the unconditional laws of $L$, as previously assumed in the literature, then the key tool is represented by an equivalent probability measure which makes the random variable $L$ independent of the original filtration $\FF$ and under which every $\FF$-martingale is also a $\GG$-martingale. 
The idea of such a measure, called {\em martingale preserving probability measure} in \cite{ABS}, goes back to early works on enlargement of filtrations and also appears in \cite{FI93}.
Together with Girsanov's theorem, this measure permits to easily move between $\FF$ and $\GG$, thus allowing to transfer the predictable representation property from $\FF$ onto $\GG$. In contrast, if Jacod's density hypothesis is only assumed to hold in the absolutely continuous form of \cite{Jac85}, then a martingale preserving probability measure may not exist (at least on the original probability space) and one has to rely on a different methodology.

Referring to Section \ref{outline} for a more detailed presentation, let us briefly describe our approach to a general martingale representation in initially enlarged filtrations, assuming the validity of Jacod's density hypothesis as stated in \cite{Jac85} and the existence of an $\FF$-local martingale having the martingale representation property in $\FF$.
First, as a preliminary step, we shall study the general structure of the initially enlarged filtration $\GG$, establishing its right-continuity and a useful characterization of $\GG$-martingales in terms of parameterized families of $\FF$-martingales. As a second step, we obtain a representation result which holds simultaneously for all the elements of a parameterized family of $\FF$-martingales. Finally, by relying on the results of \cite{SY} on stochastic integration depending on a parameter, we go back to the enlarged filtration $\GG$ in order to obtain the desired martingale representations. In this last step, a crucial ingredient is represented by the $\GG$-optional decomposition of $\FF$-local martingales recently established in \cite{ACJ15} together with the results of \cite{AFK,ACJ15} on the behavior of $\FF$-local martingales in initially enlarged filtrations.

The approach adopted in the present paper to establish the martingale representation property in initially enlarged filtrations crucially exploits Jacod's density hypothesis. However, some of the techniques used in the proofs have already appeared in the theory of enlargement of filtrations. In particular, referring to the following sections for more precise references to the literature, the results stated in Section \ref{subsec1} on the structure of $\GG$-martingales are related to similar results in \cite{CJZ,JLC09,EKJJ,GVV06}.
Moreover, even though a martingale preserving probability measure does not necessarily exist on the original probability space, we show that there exists a process playing a similar role and providing a precise link between $\FF$-martingales and $\GG$-martingales (see Remark \ref{rem:aux_repr}).
This latter fact is also related to the {\em local solution method} developed in \cite{Song_thesis}, based on the insight of constructing locally a martingale preserving probability measure on an auxiliary probability space. By combining measure changes with filtration changes, the local solution method can be shown to provide a general approach to  enlargement of filtration problems (see \cite{Song_local} for a recent account).
Since Jacod's density hypothesis can be embedded in the local solution method (see \cite[Section 6.1]{Song_local}), that method could also provide a strategy, alternative to the one adopted in the present paper, to prove the martingale representation property in initially enlarged filtrations.
The local solution method has been adopted in \cite{JS15} to establish general results on the validity of the strong predictable representation property in the case of filtrations enlarged progressively (and not initially) with respect to a non-negative random variable.
We also mention that, when $\FF$ is progressively enlarged with a non-negative random variable satisfying Jacod's density hypothesis and under the additional assumption that all $\FF$-martingale are continuous, a martingale representation result has been obtained in \cite{JLC2010}.

The paper is structured as follows. Section \ref{results} presents the probabilistic framework, the statements of the main results and two examples. In particular, we give several alternative martingale representation results depending on whether and how the conditional densities of the random variable $L$ are allowed to reach zero (see Subsection \ref{sec:main_results}). Section \ref{sec:hedging} presents an application to the hedging of contingent claims under insider information. Section \ref{proofs} contains all the proofs of the results stated in Section \ref{results} as well as several auxiliary results. Finally, the Appendix contains an alternative approach to Subsection \ref{subsec2}.

\section{Setting and main results}	\label{results}

\subsection{Notation and preliminaries}

In this paper we shall be working on several stochastic bases. Hence, we  introduce the following notation for a generic probability space $(\Omega',\mathcal{A}',\prob')$ endowed with a filtration $\FF'=(\F_t')_{t\geq0}$ satisfying the usual conditions of right-continuity and $\prob'$-completeness. We refer to \cite{MR1943877} and \cite{MR1219534} for all unexplained notions related to stochastic processes and stochastic integration.
\begin{itemize}
\item $\M(\prob',\FF')$ ($\Mloc(\prob',\FF')$, resp.) denotes the set of all martingales (local martingales, resp.) on $(\Omega',\FF',\prob')$ (in view of \cite[Remark I.1.37]{MR1943877}, we shall always assume that local martingales have c\`adl\`ag paths);
\item
if $X=(X_t)_{t\geq0}$ is an $\Real^d$-valued process in $\Mloc(\prob',\FF')$, we denote by $L_m(X;\prob',\FF')$ the set of all $\Real^d$-valued $\FF'$-predictable processes which are integrable with respect to $X$ under the measure $\prob$ in the sense of local martingales; 
\item
if $X=(X_t)_{t\geq0}$ is an $\Real^d$-valued $\FF'$-semimartingale, we denote by $L(X;\prob',\FF')$ the set of all $\Real^d$-valued $\FF'$-predictable processes which are integrable with respect to $X$ under the measure $\prob$ in the sense of semimartingales.
\end{itemize} 

Adopting the notation of \cite{MR1943877}, we denote by $(H\cdot X)_t:=\int_{(0,t]}H_u\ud X_u$ the stochastic integral of $H$ with respect to $X$, for all $t\in\Real_+$, with $(H\cdot X)_0=0$. 
We denote by $\cO(\FF')$ and by $\cP(\FF')$ the $\FF'$-optional and the $\FF'$-predictable $\sigma$-fields, respectively, on $\Omega'\times\Real_+$.

Let us recall the notion of strong predictable representation property (see \cite[Chapter IV]{MR542115} as well as \cite[Chapter 13]{MR1219534} for the one-dimensional case), here formulated with respect to an $\Real^d$-valued local martingale $X=(X_t)_{t\geq0}$ on a generic filtered probability space $(\Omega',\mathcal{A}',\FF',\prob')$. 

\begin{defn}	\label{def:PRP}
A local martingale $X=(X_t)_{t\geq0}$ is said to have the \emph{strong predictable representation property} on $(\Omega',\FF',\prob')$ if
\[
\Mloc(\prob',\FF') = \left\{\zeta + \varphi\cdot X : \zeta\in L^0(\F'_0)\text{ and }\varphi\in L_m(X;\prob',\FF')\right\},
\]
with $L^0(\F'_0)$ denoting the set of all $\F_0'$-measurable random variables.
\end{defn}

In other words, a local martingale $X$ has the strong predictable representation property on $(\Omega',\FF',\prob')$ if and only if every local martingale on that space null at zero can be written as a stochastic integral with respect to $X$. In this case, the local martingale $X$ is also said to have the \emph{(strong) martingale representation property} on $(\Omega',\FF',\prob')$.

\subsection{Setting}	\label{subsec:setting}

As the first main ingredient of our framework, we consider a probability space $(\Omega,\A,\prob)$ endowed with a filtration $\bF = (\F_t)_{t\geq0}$ satisfying the usual conditions of right-continuity and $\prob$-completeness. 
We do not necessarily assume that the initial $\sigma$-field $\F_0$ is trivial and, in general, $\bigvee_{t\geq0}\F_t=\F_{\infty-}\subseteq\A$, with the inclusion being potentially strict.
We also let $S=(S_t)_{t\geq0}$ be a given $\Real^d$-valued local martingale on $\basisp$. 

As the second main ingredient of our framework, we consider an $\A$-measurable random variable $L$ taking values in a Lusin space $(E,\cE)$, where $\cE$ denotes the Borel $\sigma$-field of $E$. Let $\lambda : \cE \rightarrow [0,1]$ be the (unconditional) law of $L$, so that $\lambda(B) = \prob(L \in B)$ holds for all $B \in \cE$. 
We then enlarge the filtration $\FF$ by adding the information of the random variable $L$ to the initial $\sigma$-field $\F_0$, i.e., we consider the filtration $\GG=(\G_t)_{t\geq0}$ given by the right-continuous augmentation of the filtration $\GG^0=(\G^0_t)_{t\geq0}$ defined as $\G^0_t:=\F_t\vee\sigma(L)$, for all $t\geq0$. 

As an auxiliary tool, let us also introduce the product space $(\widehat{\Omega},\oA,\oPP)$ via
\[
\widehat{\Omega} := E\times\Omega,
\qquad
\oA := \cE\otimes\A,
\qquad
\oPP := \lambda\otimes\PP,
\]
equipped with the right-continuous filtration $\oFF = (\oF_t)_{t\geq0}$, defined by $\oF_t := \bigcap_{\varepsilon>0}(\cE\otimes\F_{t+\varepsilon})$.
We let $\cO(\oFF)$ and $\cP(\oFF)$ be the optional and predictable $\sigma$-fields, respectively, associated to the filtration $\oFF$. 
As remarked in \cite{Jac85}, it holds that $\cE\otimes\cO(\FF)\subseteq\cO(\oFF)$ and $\cE\otimes\cP(\FF)=\cP(\oFF)$. Moreover, if $A\in\cO(\oFF)$, then the section $A_x:=\{(\omega,t)\in \Omega\times\Real_+ : (x,\omega,t)\in A\}$ is $\cO(\FF)$-measurable, for every $x\in E$.
Finally, we denote by $\oE[\cdot]$ the expectation operator with respect to the product measure $\oPP$.

\subsection{The conditional densities of $L$}

For all $t\in\Real_+$, let $\nu_t : \Omega \times \cE \rightarrow [0,1]$ be a regular version of the $\F_t$-conditional law of $L$ (which always exists since the space $(E,\cE)$ is Lusin).
The following assumption will play a central role in our analysis.

\begin{ass}	\label{ass:Jacod}
For all $t\in\Real_+$, $\nu_t \ll\lambda$ holds in the $\PP$-a.s. sense.
\end{ass}

Assumption~\ref{ass:Jacod} corresponds to the classical density hypothesis introduced in \cite{Jac85}. Indeed, as shown in \cite[Proposition 1.5]{Jac85} (see also \cite[Theorem VI.11]{MR1037262}), Assumption~\ref{ass:Jacod} holds if and only if, for all $t\in\Real_+$, there exists a positive $\sigma$-finite measure $\gamma_t$ on $(E,\cE)$ such that $\nu_t \ll \gamma_t$ holds in the $\PP$-a.s. sense. 

The following lemma gives the existence of a good version of the conditional densities and essentially corresponds to \cite[Lemma 1.8]{Jac85} (see Section \ref{proofs} for a proof). 

\begin{lem}	\label{lem-Jac}
Suppose that Assumption \ref{ass:Jacod} holds.
Then there exists a $\pare{\cE \otimes \cO(\FF)}$-measurable function $E \times \Omega \times \Real_+ \ni (x, \omega,t)\mapsto q^x_t(\omega) \in \Real_+$, c\`adl\`ag \ in $t \in\Real_+$ and such that:
\begin{itemize}
\item[(i)]
for every $t\in\Real_+$, $\nu_t(\ud x) = q^x_t \,\lambda(\ud x)$ holds $\prob$-a.s;
\item[(ii)]
for every $x\in E$, the process $q^x=(q^x_t)_{t\geq0}$ is a martingale on $\basisp$.
\end{itemize}
\end{lem}

For every $x\in E$ and $n\in\N$, let us define the following $\FF$-stopping times:
\[
\zeta^x_n:=\inf\{t\in\R_+ \such q^x_t<1/n\}
\qquad\text{and}\qquad
\zeta^x:=\inf\{t\in\R_+ \such q^x_t=0\}.
\]
For every $x\in E$, it holds that $\{\zeta^x_n\}_{\nin}$ is a nondecreasing sequence, $\prob(\limn \zeta^x_n = \zeta^x) = 1$, and $q^x=0$ on $\dbraco{\zeta^x,\infty}$ (see also \cite[Lemma 1.8]{Jac85}). Note also that, due to \cite[Corollary 1.11]{Jac85}, it holds that $\prob(\zeta^L < \infty) = 0$, with $\zeta^L(\omega) \dfn \zeta^{L(\omega)}(\omega)$. As in \cite{ACJ15,AFK}, for every $x\in E$, we consider the $\F_{\zeta^x}$-measurable event $\Lambda^x:=\{\zeta^x<\infty,q^x_{\zeta^x-}>0\}$ and define
\be	\label{eta-init}
\eta^x := \zeta^x_{\Lambda^x} = \zeta^x\ind_{\Lambda^x}+\infty\ind_{\Omega\setminus\Lambda^x}, 
\ee
which is an $\FF$-stopping time and represents the time at which $q^x$ jumps to zero.

It is a fundamental result of \cite{Jac85} that, under Assumption \ref{ass:Jacod}, the canonical decomposition  in $\GG$ of an arbitrary $\FF$-local martingale can be written in terms of the conditional densities (evaluated at $x=L$). However, a different type of decomposition of $\FF$-local martingales in $\GG$ turns out to be better suited to our analysis. The following proposition has been recently established in \cite[Theorem 5]{ACJ15} and provides an optional decomposition (as opposed to the canonical decomposition) of the $\FF$-local martingale $(S_t)_{t\geq0}$ in $\GG$.

\begin{prop}	\label{prop:Aks}
Suppose that Assumption \ref{ass:Jacod} holds and that the space $L^1(\Omega,\A,\prob)$ is separable. Then the process $S^{\GG}=(S^{\GG}_t)_{t\geq0}$ defined as
\be	\label{eq:dec_Aks}
S^{\GG} := S - \frac{1}{q^L}\cdot [S,q^L] + \left(\Delta S_{\eta^x}\ind_{\dbraco{\eta^x,\infty}}\right)^{p,\FF}\bigr|_{x=L}
\ee
is a local martingale on $(\Omega,\GG,\prob)$, with $(\Delta S_{\eta^x}\ind_{\dbraco{\eta^x,\infty}})^{p,\FF}$ denoting a $\cE$-measurable version of the dual $\FF$-predictable projection of the process $\Delta S_{\eta^x}\ind_{\dbraco{\eta^x,\infty}}$.
\end{prop}

\begin{rem}	\label{rem:sep}
The separability assumption appearing in Proposition \ref{prop:Aks} is only needed to ensure the existence of a version of the dual $\FF$-predicable projection of $\Delta S_{\eta^x}\ind_{\dbraco{\eta^x,\infty}}$ which is measurable in $x$, see \cite[Proposition 4]{SY}.\footnote{As remarked in \cite[Remark 1 after Proposition 4]{SY}, this separability assumption is always verified in practice and it is not useful to consider further generalizations.} 
\end{rem}

\subsection{The strong predictable representation property in $\GG$}	\label{sec:main_results}

This section contains the statements of the main results. Since the proofs are rather technical and involve several intermediate steps and auxiliary results, they are postponed to Section \ref{proofs}. We start with the following theorem, which represents the central result, referring to Section \ref{outline} for an outline of its proof.

\begin{thm}	\label{thm:main}
Suppose that Assumption \ref{ass:Jacod} holds and that the space $L^1(\Omega,\A,\prob)$ is separable. If $S=(S_t)_{t\geq0}$ has the strong predictable representation property on $(\Omega,\FF,\prob)$, then the process $S^{\GG}=(S^{\GG}_t)_{t\geq0}$ defined in \eqref{eq:dec_Aks} has the strong predictable representation property on $(\Omega,\GG,\prob)$.
\end{thm}

In view of Definition \ref{def:PRP}, Theorem \ref{thm:main} shows that, if $S$ has the strong predictable representation property on $\basisp$, then every local martingale $M=(M_t)_{t\geq0}$ on $\basisgp$ admits the stochastic integral representation
\be	\label{eq:repr_G}
M_t = M_0 + (\varphi\cdot S^{\GG})_t,
\qquad \text{ $\prob$-a.s. for all }t\in\Real_+,
\ee
where $(\varphi_t)_{t\geq0}$ is a $\GG$-predictable process admitting a rather explicit characterization (see \eqref{eq:integrand}).
As mentioned in the introduction, Theorem \ref{thm:main} generalizes the martingale representation results previously obtained in the literature on initially enlarged filtrations. In particular, no assumption is made on the family $\{q^x : x\in E\}$ of conditional densities of $L$, apart from its existence.   

In the remaining part of this subsection, we present some alternative martingale representation results under additional assumptions on the conditional densities of $L$. We start with the following corollary (proved in Section \ref{proofs_main}) which shows that, if the $\FF$-martingale $q^x$ can only reach zero continuously (and not due to a jump), for $\lambda$-a.e. $x\in E$, then the strong predictable representation property of $S$ on $\basisp$ can be easily transferred to $\basisgp$ up to a suitable ``change of num\'eraire'' with respect to the process $(q^L_t)_{t\geq0}$.

\begin{cor}	\label{cor:no_jump}
Suppose that Assumption \ref{ass:Jacod} holds and that $\prob(\eta^x<\infty)=0$ for $\lambda$-a.e. $x\in E$. If the process $S=(S_t)_{t\geq0}$ has the strong predictable representation property on $(\Omega,\FF,\prob)$, then the $\Real^{d+1}$-valued process $(1/q^L_t,S_t/q^L_t)_{t\geq0}$ is a $\GG$-local martingale and  has the strong predictable representation property on $(\Omega,\GG,\prob)$.
\end{cor}

\begin{rem}	\label{rem:num}
Theorem \ref{thm:main} and Corollary \ref{cor:no_jump} show that, as long as $\prob(\eta^x<\infty)=0$ for $\lambda$-a.e. $x\in E$, the martingale representation property in $\GG$ can be expressed in terms of both $S^{\GG}$ and $(1/q^L,S/q^L)$. The representation result of Theorem \ref{thm:main} is obviously more general and holds with respect to a $d$-dimensional process (i.e., of the same dimension of the original process $S$). The representation result of Corollary \ref{cor:no_jump} is less general and requires a $(d+1)$-dimensional process.\footnote{The precise relation between $S^{\GG}$ and $(1/q^L,S/q^L)$ is shown in the proof of Corollary \ref{cor:no_jump}, given in Section \ref{proofs_main}.} 
However, the process $(1/q^L,S/q^L)$ admits an important interpretation, especially in the context of financial modeling. Indeed, under the assumptions of Corollary \ref{cor:no_jump} and in view of \cite{AFK}, the process $q^L$ represents the {\em num\'eraire portfolio} for $S$ in the enlarged filtration $\GG$ (see \cite{MR2335830}).  In this sense, Corollary \ref{cor:no_jump} shows that the martingale representation property can be transferred from $\FF$ onto $\GG$ by changing the num\'eraire, choosing the num\'eraire portfolio in $\GG$ as the baseline asset.
Note also that, in view of applications, the process $(1/q^L,S/q^L)$ can be immediately deduced from a model's fundamental ingredients and does not require any computation, unlike the process $S^{\GG}$ appearing in Theorem \ref{thm:main}. 
\end{rem}

As mentioned in the introduction, the existing martingale representation results in initially enlarged filtrations have been obtained under the stronger assumption that $\nu_t\sim\lambda$ (see \cite[Theorem 4.3]{GP98}, \cite[Theorem 4.2]{Amen}, \cite[Theorem 3.2]{ABS} and \cite[Proposition 5.3]{CJZ}). 
This case corresponds to the following proposition, which can be easily deduced from our general approach, as shown in Section \ref{proofs_main}.

\begin{prop}	\label{prop:classic_case}
Suppose that $\nu_t\sim\lambda$ holds $\prob$-a.s. for all $t\in[0,T]$, for some fixed $T<\infty$.  
If $S=(S_t)_{t\in[0,T]}$ has the strong predictable representation property on $(\Omega,\FF,\prob)$, then $\ud\ptilde:=(q^L_0/q^L_T)\ud\prob$ defines a probability measure $\ptilde\sim\prob$ such that $S\in\mMloc(\ptilde,\GG)$ and $S$ has the strong predictable representation property on $(\Omega,\GG,\ptilde)$.
\end{prop}

We close this subsection with a last martingale representation result, under the same assumptions of Corollary \ref{cor:no_jump}. As a preliminary, we recall that, in view of \cite[Theorem 2.5]{Jac85}, the process $\langle S,q^x\rangle^{\FF}\bigr|_{x=L}$ is well-defined\footnote{Unlike the quadratic variation $[\cdot,\cdot]$, the predictable quadratic variation depends on the filtration. By \cite[Theorem 2.5]{Jac85}, there exists a $(\cE\otimes\cP(\FF))$-measurable version of the map $(x,\omega,t)\mapsto \langle S,q^x\rangle_t^{\FF}(\omega)$, which is well-defined on the set $\{q^x_->0\}$. Hence, $\langle S,q^x\rangle^{\FF}\bigr|_{x=L}$ denotes the $\FF$-predictable quadratic variation evaluated at $x=L$.} and the process $S-\frac{1}{q^L_-}\cdot\langle S,q^x\rangle^{\FF}\bigr|_{x=L}$ is a local martingale on $\basisgp$.
The following result (proved in Section \ref{proofs_main}) has been established in \cite[Proposition 5.5]{CJZ} under the stronger assumption that $\nu_t\sim\lambda$ $\prob$-a.s. for all $t\in\Real_+$.

\begin{cor}	\label{cor:no_jump_2}
Suppose that Assumption \ref{ass:Jacod} holds and that $\prob(\eta^x<\infty)=0$ for $\lambda$-a.e. $x\in E$. If $S=(S_t)_{t\geq0}$ has the strong predictable representation property on $(\Omega,\FF,\prob)$, then the process $\bar{S}^{\GG}:=S-\frac{1}{q^L_-}\cdot\langle S,q^x\rangle^{\FF}\bigr|_{x=L}$ has the strong predictable representation property on $\basisgp$.
\end{cor}

We can observe that, while Corollary \ref{cor:no_jump} relates the martingale representation property to a change of num\'eraire (see Remark \ref{rem:num}), Corollary \ref{cor:no_jump_2} relates the martingale representation property to a locally equivalent change of measure, as made explicit by the proof given in Section \ref{proofs_main}.

\subsection{Two examples}	\label{sec:example}

We now present two simple examples of processes having the strong predictable representation property in the initially enlarged filtration $\GG$. For simplicity, we consider a fixed time horizon $T<\infty$. In both examples, Assumption \ref{ass:Jacod} is satisfied but the equivalence $\nu_t\sim\lambda$  does not hold. Hence, the following martingale representation results are not covered by the existing literature. We first present an example where the conditional densities can reach zero due to a jump and then an example where the conditional densities can only reach zero continuously (i.e., $\prob(\eta^x<\infty)=0$ for $\lambda$-a.e. $x\in E$). 
In a financial context, with $S$ representing the price of a risky asset (see Section \ref{sec:hedging}), these two examples admit interesting interpretations in relation to the modeling of insider information and have been considered in  \cite{AFK} and \cite{CT15}, respectively.

\subsubsection{An example based on the Poisson process}	\label{sec:poisson}

Let $N=(N_t)_{t\in[0,T]}$ be a standard Poisson process and $\FF=(\F_t)_{t\in[0,T]}$ its $\prob$-augmented natural filtration, denoting by $\{\tau_n\}_{n\in\N}$ the jump times of $N$. We let $\A=\F_T$ and consider the random variable $L=N_T$ together with the corresponding initially enlarged filtration $\GG$. Similarly as in \cite[Example 1.5.3]{AFK} (see also \cite[\textsection 4.3]{GVV06}), it can be easily checked that, for every $n\in\N$,
\[
q^n_t = \frac{\prob(L=n|\F_t)}{\prob(L=n)}
= e^t\frac{(T-t)^{n-N_t}}{T^n}\frac{n!}{(n-N_t)!}\ind_{\{N_t\leq n\}},
\qquad\text{ for all }t<T,
\]
and $q^n_T=e^TT^{-n}n!\ind_{\{N_T=n\}}$, thus showing that Assumption \ref{ass:Jacod} is satisfied. Observe also that the conditional density $(q^n_t)_{t\in[0,T]}$ jumps to zero at $\tau_{n+1}$ (if $\tau_{n+1}\leq T$), meaning that $\eta^n=\tau_{n+1}$, for all $n\in\N$. 

The compensated Poisson process $S:=(N_t-t)_{t\in[0,T]}$ has the strong predictable representation property on $\basisp$ (see e.g. \cite[Proposition 8.3.5.1]{JYC09}). Hence, the assumptions of Theorem \ref{thm:main} are satisfied and the process $S^{\GG}=(S^{\GG}_t)_{t\in[0,T]}$ defined in \eqref{eq:dec_Aks} has the strong predictable representation property on $\basisgp$. Moreover, $S^{\GG}$ can be explicitly represented as follows:
\be	\label{eq:SG_explicit}
S^{\GG}_t = (N_t-t)
- \sum_{n=1}^{N_T}\ind_{\{\tau_n\leq t\}}\left(1-\frac{T-\tau_n}{N_T-n+1}\right)
+ (t-\tau_{N_T}\wedge t).
\ee
Indeed, noting that $[S,q^L]=\sum_{0<u\leq\cdot}\Delta N_u\Delta q^L_u=\sum_{n=1}^{N_T}\Delta q^L_{\tau_n}\ind_{\{\tau_n\leq\cdot\}}$ and that
\[
\Delta q^L_{\tau_n}
= e^{\tau_n}\frac{(T-\tau_n)^{N_T-n}}{T^{N_T}}\frac{N_T!}{(N_T-n)!}\left(1-\frac{T-\tau_n}{N_T-n+1}\right),
\]
it holds that
\[
\frac{1}{q^L}\cdot[S,q^L]
= \sum_{n=1}^{N_T}\ind_{\{\tau_n\leq \cdot\}}\left(1-\frac{T-\tau_n}{N_T-n+1}\right).
\]
Moreover, observe that $\Delta S_{\eta^n}\ind_{\dbra{\eta^n,T}}=\ind_{\dbra{\tau_{n+1},T}}$, for all $n\in\N$. Hence, in view of Proposition \ref{prop:Aks}, in order to prove \eqref{eq:SG_explicit} it remains to compute the dual $\FF$-predictable projection of the increasing process $(\ind_{\{\tau_{n+1}\leq t\}})_{t\in[0,T]}$. Observe that $\{\tau_{n+1}\leq t\}=\{N_t\geq n+1\}$, so that $\ind_{\{\tau_{n+1}\leq t\}}=f^n(N_t)$, with $f^n(x):=\ind_{\{x\geq n+1\}}$.
By \cite[Proposition 8.2.3.1]{JYC09}, the infinitesimal generator $\mathcal{L}$ of $N$ is given by $\mathcal{L}(f)(\cdot)=f(\cdot+1)-f(\cdot)$, for any bounded measurable function $f$. Hence, the process
\[	\ba
f^n(N_t) - \int_0^{t}\mathcal{L}(f^n)(N_u)\ud u
&= \ind_{\{\tau_{n+1}\leq t\}} - \int_0^{t}\ind_{\{n\leq N_u< n+1\}}\ud u
= \ind_{\{\tau_{n+1}\leq t\}} - (\tau_{n+1}\wedge t-\tau_n\wedge t)
\ea	\]
is a local martingale on $\basisp$, for each $n\in\N$. By the Doob-Meyer decomposition of the bounded submartingale $(\ind_{\{\tau_{n+1}\leq t\}})_{t\in[0,T]}$, it is also a uniformly integrable martingale on $\basisp$ (see e.g. \cite[Theorem I.3.15]{MR1943877}). By \cite[Corollary 5.31]{MR1219534}, this implies that $(\tau_{n+1}\wedge\cdot-\tau_n\wedge\cdot)$ is the dual $\FF$-predictable projection of $(\ind_{\{\tau_{n+1}\leq t\}})_{t\in[0,T]}$, for all $n\in\N$, thus proving \eqref{eq:SG_explicit}.

\subsubsection{An example based on the Brownian motion}

As in \cite[Section 5.1]{CT15}, let define the process $S=(S_t)_{t\in[0,T]}$ by $S:=1+W$, where $W=(W_t)_{t\in[0,T]}$ is a standard Brownian motion and $\FF=(\F_t)_{t\in[0,T]}$ its $\prob$-augmented natural filtration. We let $\A=\F_T$ and consider the discrete random variable $L=\ind_{\{\inf_{t\in[0,T]}S_t>0\}}$. Since $L$ is a discrete random variable, Assumption \ref{ass:Jacod} holds (see e.g. \cite[Corollary 2 to Theorem VI.10]{MR1037262}) and, following \cite{CT15}, we can compute
\[	\ba
q^1_t &= \frac{\prob(\inf_{u\in[0,T]}S_u>0|\F_t)}{\prob(\inf_{u\in[0,T]}S_u>0)}	
= 1+\frac{1}{\prob(\inf_{u\in[0,T]}S_u>0)}\sqrt{\frac{2}{\pi}}\int_0^{\sigma\wedge t}\frac{1}{\sqrt{T-s}}e^{-\frac{S^2_s}{2(T-s)}}\ud W_s,\\
q^0_t &= \frac{\prob(\inf_{u\in[0,T]}S_t\leq0|\F_t)}{\prob(\inf_{u\in[0,T]}S_u\leq0)}
= 1-\frac{1}{\prob(\inf_{u\in[0,T]}S_u\leq0)}\sqrt{\frac{2}{\pi}}\int_0^{\sigma\wedge t}\frac{1}{\sqrt{T-s}}e^{-\frac{S^2_s}{2(T-s)}}\ud W_s,
\ea	\]
where $\sigma:=\inf\{t\in\R_+ \such S_t=0\}$.
Observe that $q^1$ and $q^0$ can reach zero in a continuous way with a strictly positive probability. Since $W$ has the strong predictable representation property on $\basisp$, the assumptions of Corollary  \ref{cor:no_jump_2} are satisfied. Noting that $q^L_t=q^1_t\ind_{\{\sigma>T\}}+q^0_t\ind_{\{\sigma\leq T\}}$,  the strong predictable representation property on $\basisgp$ holds with respect to the process
\[	\ba
\bar{S}^{\GG}_t &= S_t - \int_0^t\frac{1}{q^L_s}\ud\langle S,q^x\rangle^{\FF}_s\bigr|_{x=L}	\\
&= S_t 
- \frac{\ind_{\{\inf_{u\in[0,T]}S_u>0\}}}{\prob(\inf_{u\in[0,T]}S_u>0)}\sqrt{\frac{2}{\pi}}\int_0^t\frac{1}{q^1_s\sqrt{T-s}}e^{-\frac{S^2_s}{2(T-s)}}\ud s	\\
&\quad+ \frac{\ind_{\{\inf_{u\in[0,T]}S_u\leq0\}}}{\prob(\inf_{u\in[0,T]}S_u\leq0)}\sqrt{\frac{2}{\pi}}\int_0^{\sigma\wedge t}\frac{1}{q^0_s\sqrt{T-s}}e^{-\frac{S^2_s}{2(T-s)}}\ud s.
\ea	\]

\section{Hedging under insider information}	\label{sec:hedging}

In this section, we study the implications of the martingale representation property  in the context of an abstract financial market with insider information. For simplicity of presentation, we consider a fixed time horizon $T<\infty$ and an $\Real^d$-valued local martingale $S=(S_t)_{t\in[0,T]}$ having the strong predictable representation property on $\basisp$. In the present section, we also assume that the initial $\sigma$-field $\F_0$ is trivial.
As in Section \ref{subsec:setting}, we consider a random variable $L$ together with the associated initially enlarged filtration $\GG$ and suppose that Assumption \ref{ass:Jacod} is satisfied, so that the family $\{q^x: x\in E\}$ is well-defined and $\A:=\G_T=\F_T\vee\sigma(L)$ (see Lemma \ref{lem:RC_filtr}). 

The interpretation of the above setting is as follows. The process $S$ represents the prices (discounted with respect to some baseline security) of $d$ risky assets traded in the market and the filtration $\FF$ represents the information publicly available to every market participant. On the contrary, the random variable $L$ represents an additional information (\emph{insider information}) which is only available to some better informed agents, having access to the information flow $\GG$. The better informed agents are allowed to trade on the same set of securities as the uninformed agents but can rely on their private information when choosing their strategies.
For $\mathbf{H}\in\{\FF,\GG\}$, we denote by $\mathcal{H}(\mathbf{H})$ the set of admissible strategies based on the information flow $\mathbf{H}$, i.e.,
\[
\mathcal{H}(\mathbf{H}) := \bigl\{H\in L(S;\prob,\mathbf{H}) : (H\cdot S)_t\geq-a \text{ $\prob$-a.s. for all }t\in[0,T],\text{ for some }a\in\Real_+\bigr\},
\]
which amounts to exclude trading strategies requiring an unlimited line of credit.

One of the fundamental problems in mathematical finance is represented by the hedging of contingent claims. We interpret any $\A$-measurable non-negative random variable $\xi$ as a \emph{contingent claim}. The hedging problem, with respect to a filtration $\mathbf{H}\in\{\FF,\GG\}$, consists in finding a strategy $H\in\mathcal{H}(\mathbf{H})$ such that $\xi=v^{\mathbf{H}}(\xi)+(H\cdot S)_T$ holds $\prob$-a.s., for some initial wealth $v^{\mathbf{H}}(\xi)$. If this is possible, then the strategy $H$ is said to \emph{replicate} the contingent claim $\xi$ and $v^{\mathbf{H}}(\xi)$ represents the initial cost of replicating $\xi$ having access to the information flow $\mathbf{H}$. If every bounded contingent claim can be replicated, then the financial market is said to be \emph{complete}.
Besides its own interest, the hedging problem has fundamental applications in mathematical finance, notably in the context of pricing and portfolio optimization.\footnote{In a forthcoming paper, we shall apply the present results to utility maximization problems in the presence of insider information which can generate arbitrage opportunities, by relying on the duality approach of \cite{CCFM}.}

Since $S$ has the strong predictable representation property on $\basisp$, it is well-known that every contingent claim $\xi\in L^1_+(\F_T)$ can be replicated. Indeed, it suffices to consider the non-negative $\FF$-martingale $M=(M_t)_{t\in[0,T]}$ defined by $M_t:=\expec[\xi|\F_t]$, for all $t\in[0,T]$. Then, in view of Definition \ref{def:PRP}, there exists an $\Real^d$-valued process $H\in L_m(S;\prob,\FF)$ such that $\xi=M_T=M_0+(H\cdot S)_T$ $\prob$-a.s. Moreover, it holds that $(H\cdot S)_t=M_t-M_0\geq-\expec[\xi]$, thus showing that $H\in\mathcal{H}(\FF)$. The initial cost of replicating $\xi$ is given by $v^{\FF}(\xi)=M_0=\expec[\xi]$. This solves the hedging problem of an $\F_T$-measurable contingent claim from the perspective of an uninformed agent having access to the information flow $\FF$. Clearly, there does not exist in general an $\FF$-predictable hedging strategy for a $\G_T$-measurable contingent claim.

The following proposition is the central result of this section and shows that Assumption \ref{ass:Jacod} together with the completeness of the financial market based on $\basisp$ suffices to ensure that the financial market based on $\basisgp$ is also complete (up to a $\sigma(L)$-measurable initial wealth).

\begin{prop}	\label{prop:hedging}
Suppose that Assumption \ref{ass:Jacod} holds and that $S=(S_t)_{t\in[0,T]}$ has the strong predictable representation property on $\basisp$. Let $\xi$ be a bounded non-negative $\G_T$-measurable random variable. Then there exists a strategy $H\in\mathcal{H}(\GG)$ which replicates $\xi$ with initial cost $v^{\GG}(\xi)=\expec[\xi/q^L_T|\sigma(L)]$.
\end{prop}
\begin{proof}
Let $M=(M_t)_{t\in[0,T]}$ be the non-negative $\GG$-martingale defined by $M_t:=\expec[\xi/q^L_T|\G_t]$, for all $t\in[0,T]$.  By Proposition \ref{lem:repr_1} below (recalling that $q^x_0=1$ for all $x\in E$, since in this section $\F_0$ is assumed to be trivial), there exists a $\GG$-predictable process $K^L$ such that
\[
q^L_t\expec\left[\frac{\xi}{q^L_T}\Bigr|\G_t\right]
= q^L_tM_t
= M_0 + (K^L\cdot S)_t,
\qquad\text{ $\prob$-a.s. for all }t\in[0,T],
\]
where $K^L\in L(S;\prob,\GG)$. Since $\xi\geq0$ $\prob$-a.s., it holds that $(K^L\cdot S)_t\geq-M_0$ $\prob$-a.s. for all $t\in[0,T]$. Moreover, the boundedness of $\xi$ together with the supermartingale property of the process $(1/q^L_t)_{t\in[0,T]}$ (see \cite[Proposition 3.4]{AFK} or \cite[Lemma 5]{ACJ15}) yields that $M_0\leq C$ $\prob$-a.s., for some $C\in\Real_+$, thus showing  that $K^L\in\mathcal{H}(\GG)$. Evaluating the above expression for $t=T$ yields $\xi = \expec[\xi/q^L_T|\sigma(L)]+(K^L\cdot S)_T$ $\prob$-a.s., thus proving the claim. 
\end{proof}

In particular, the above proposition can be applied to an $\F_T$-measurable contingent claim $\xi$. In this case, both for the uninformed agent and for the informed agent there exists a hedging strategy. However, since the two agents have access to different information flows ($\FF$ and $\GG$, respectively), the hedging strategy is not necessarily the same and the initial cost of replicating $\xi$ will depend on the available information. This is the content of the following corollary, which shows that the better informed agent can always take advantage of the insider information and replicate any $\F_T$-measurable contingent claim at a lower cost.

\begin{cor}
Suppose that Assumption \ref{ass:Jacod} holds and that $S=(S_t)_{t\in[0,T]}$ has the strong predictable representation property on $\basisp$.
Then, for every bounded non-negative $\F_T$-measurable random variable $\xi$, it holds that $v^{\GG}(\xi)\leq v^{\FF}(\xi)$.
\end{cor}
\begin{proof}
Let $g:E\rightarrow\Real$ be a bounded $\cE$-measurable function. Formula \eqref{expec_init} below then gives
\[
\expec\left[g(L)\frac{\xi}{q^L_T}\right]
= \expec\left[g(L)\frac{\xi}{q^L_T}\ind_{\{q^L_T>0\}}\right]
=\int_Eg(x)\expec[\xi\ind_{\{q^x_T>0\}}]\lambda(\ud x)
= \expec\left[g(L)\expec\bigl[\xi\ind_{\{q^x_T>0\}}\bigr]\bigr|_{x=L}\right],
\]
where the expectation $\expec[\xi\ind_{\{q^x_T>0\}}]$ is $\cE$-measurable by Fubini's theorem. By the arbitrariness of the function $g(\cdot)$, this implies that $\expec[\xi/q^L_T|\sigma(L)]=\expec[\xi\ind_{\{q^x_T>0\}}]|_{x=L}$. The claim then follows by Proposition \ref{prop:hedging}, noting that $\expec[\xi\ind_{\{q^x_T>0\}}]\leq \expec[\xi]=v^{\FF}(\xi)$, for every $x\in E$.
\end{proof}

In the context of the above corollary, for a given contingent claim $\xi$, the difference between the two values $v^{\FF}(\xi)$ and $v^{\GG}(\xi)$ can be regarded as a monetary value of the additional information contained in the random variable $L$ when replicating $\xi$.

\begin{rem}[On the possibility of arbitrage]	\label{rem:arb}
Since $S\in\Mloc(\prob,\FF)$, the market where $S$ is traded on the basis of the information flow $\FF$ does not admit arbitrage (in the sense of \emph{no free lunch with vanishing risk}, see \cite{MR1304434}). However, when the filtration is enlarged to $\GG$, the process $S$ might allow for free lunches with vanishing risk or even arbitrages of the first kind (as shown in \cite[Section 1.5.3]{AFK}, this is for instance the case in the example considered in Section \ref{sec:poisson}). As shown in \cite[Theorem 1.12]{AFK} and \cite[Theorem 6]{ACJ15}, the condition $\prob(\eta^x<\infty)=0$ for $\lambda$-a.e. $x\in E$ appearing in Corollaries \ref{cor:no_jump} and \ref{cor:no_jump_2} acts as a necessary and sufficient condition in order to exclude arbitrages of the first kind in the initially enlarged filtration $\GG$ for any semimartingale $S$.
However, we want to remark that the results of this section do not depend on the presence of arbitrage in $\GG$, meaning that the market can be complete even when insider information yields arbitrage.
\end{rem}

\section{Proofs and auxiliary results}		\label{proofs}

In this section, we give the proofs of the results stated in Section \ref{results} together with several auxiliary results.
We start by proving Lemma \ref{lem-Jac}, following \cite[Appendix A.1]{Amen}, which shows the existence of a good version of the conditional densities of $L$.

\begin{proof}[Proof of Lemma \ref{lem-Jac}]
By  \cite[Lemma 1.8]{Jac85}, Assumption \ref{ass:Jacod} implies the existence of an $\cO(\oFF)$-measurable nonnegative function $(x,\omega,t)\mapsto\tilde{q}^x_t(\omega)$ such that (i)-(ii) hold. Since, for every $x\in E$, the process $\tilde{q}^x$ is $\FF$-optional, being $\FF$-adapted and c\`adl\`ag,   \cite[Remark 1 after Proposition 3]{SY} gives the existence of a $\pare{\cE\otimes\cO(\FF)}$-measurable function $(x,\omega,t)\mapsto q^x_t(\omega)$ such that $(q^x_t)_{t\geq0}$ is indistinguishable from $(\tilde{q}^x_t)_{t\geq0}$, for every $x\in E$.
\end{proof}

The following consequence of Lemma~\ref{lem-Jac} will be used several times: for any $t\in\Real_+$ and $\pare{\cE\otimes\F_t}$-measurable function $E \times \Omega \ni (x, \omega)\mapsto f^x_t(\omega) \in \Real_+$, it holds that
\be	\label{expec_init}
\expec\left[f_t^L\right] 
= \expec\left[\int_Ef_t^x\,q^x_t\,\lambda(\ud x)\right]
= \int_E\expec\left[f_t^x\,q^x_t\right]\lambda(\ud x).
\ee
As usual, equality \eqref{expec_init} can be extended to integrable $\pare{\cE\otimes\F_t}$-measurable functions.

\subsection{An outline of the proof of Theorem \ref{thm:main}}	\label{outline}

The proof of Theorem \ref{thm:main} is rather technical, requires several  auxiliary results and needs to deal with some measurability issues. However, since the underlying ideas are rather simple and transparent, let us give an outline of the proof (referring to the following subsections for full details).

The first step consists in showing that Assumption \ref{ass:Jacod} already implies the right-continuity of the filtration $(\G^0_t)_{t\geq0}$, thereby extending \cite[Proposition 3.3]{Amen} (see Lemma \ref{lem:RC_filtr}). Together with formula \eqref{expec_init}, this property can be shown to imply that every process $M=(M_t)_{t\geq0}\in\M(\prob,\GG)$ can be represented as $M_t=m^L_t$, where $(x,\omega,t)\mapsto m^x_t(\omega)$ is a $(\cE\otimes\cO(\FF))$-measurable function such that $(m^x_tq^x_t)_{t\geq0}\in\M(\prob,\FF)$, for all $x\in E$ (see Proposition \ref{prop:mart}).

At this stage, if the process $S=(S_t)_{t\geq0}$ is assumed to have the strong predictable representation property on $\basisp$, this suggests to represent the $\FF$-martingales $(m^x_tq^x_t)_{t\geq0}$ and $(q^x_t)_{t\geq0}$, for each $x\in E$, as $m^x_tq^x_t=m^x_0q^x_0 + (K^x\cdot S)_t$ and $q^x_t=q^x_0+(H^x\cdot S)_t$, respectively, where $K^x$ and $H^x$ are suitable $\FF$-predictable processes, for each $x\in E$. 
Provided that the stochastic integrals $K^x\cdot S$ and $H^x\cdot S$ are measurable in $x$ and make sense in the initially enlarged filtration $\GG$, one can then evaluate the two stochastic integral representations at $x=L$. Since $q^L_t>0$ $\prob$-a.s. (see \cite[Corollary 1.11]{Jac85}), one can finally apply the integration by parts formula to $M_t=(m^L_tq^L_t)/q^L_t$ and obtain a stochastic integral representation for $M$ in $\GG$.

The main problem in the above argument is that, by applying the strong predictable representation property of $S$ to $(m^x_tq^x_t)_{t\geq0}$ and $(q^x_t)_{t\geq0}$ \emph{separately for each $x\in E$}, the $\FF$-predictable processes $K^x$ and $H^x$ appearing in the integral representations above have a priori no measurability properties with respect to $x$. Hence, one has to perform a martingale representation of $(m^x_tq^x_t)_{t\geq0}$ and $(q^x_t)_{t\geq0}$ which holds \emph{simultaneously for all} $x\in E$. To this effect, inspired by the recent paper \cite{EsmImk}, we shall work on the product space $(\widehat{\Omega},\oFF,\oPP)$ and establish a martingale representation theorem on that level (see Lemma \ref{lem:PRP_prod} and Proposition \ref{prop:mr}), thus ensuring nice measurability properties for the integrands $K^x$ and $H^x$.
One can then go back to the original space $\basisp$ and combine the results of \cite{SY} and \cite{Jac85} to prove that all stochastic integrals admit $x$-measurable versions and make sense in the initially enlarged filtration $\GG$. At the final step, we evaluate the integral representations at $x=L$ and perform an integration by parts. Together with Proposition \ref{prop:Aks}, this will give the strong predictable representation property of $S^{\GG}$ on $\basisgp$.

\begin{rem}	\label{rem:Jacod}
In the seminal paper \cite{Jac85}, the author established a (weak) martingale representation result which holds simultaneously for all $\{q^x : x\in E\}$ 
(see \cite[Proposition 3.14]{Jac85}). As detailed in the Appendix, one can apply a similar reasoning in the present setting, avoiding the introduction of the product space $(\widehat{\Omega},\oFF,\oPP)$, at the expense of relying on more sophisticated tools (notably, stochastic integration with respect to random measures and fine properties of martingale representation, as in \cite[\textsection IV.4d]{MR542115}). In comparison, our approach only uses elementary notions of stochastic calculus and basic facts on the strong predictable representation property.
\end{rem}

\subsection{The structure of the initially enlarged filtration $\GG$}	\label{subsec1}

As a preparation to the proof of the main results, we first establish some preliminary properties of the enlarged filtration $\GG$.

The first result concerns the right-continuity of the filtration $\GG^0$. 
If Assumption \ref{ass:Jacod} is replaced by the stronger assumption that $\nu_t\sim\lambda$ $\PP$-a.s. for all $t\in[0,T]$, for some fixed horizon $T<\infty$, then it is well-known that the filtration $\GG^0$ is right-continuous, see \cite[Proposition 3.3]{Amen} and \cite[Lemma 2.2]{ABS}. The following lemma extends this result and shows that Assumption \ref{ass:Jacod} actually suffices to ensure the right-continuity of $\GG^0$. \footnote{The fact that Assumption \ref{ass:Jacod} suffices to ensure the right-continuity of the filtration $\GG^0$ has been already remarked in \cite{GVV06}. However, the authors did not provide a proof of their claim. Our proof of Lemma \ref{lem:RC_filtr} is inspired from \cite[Lemma 6.8]{Song14}.} 

\begin{lem}	\label{lem:RC_filtr}
Suppose that Assumption \ref{ass:Jacod} holds. Then $\GG^0=\GG$.
\end{lem}
\begin{proof}
Let $\xi:E\times\Omega\rightarrow\Real_+$ be a $\pare{\cE\otimes\F_T}$-measurable function, for some $T\in(0,\infty)$. For any $A_t\in\F_t$, with $t\in[0,T]$, and $g:E\rightarrow\Real$ bounded $\cE$-measurable, using \eqref{expec_init} and recalling that $(q^x_t)_{t\geq0}$ is a non-negative martingale (see Lemma \ref{lem-Jac}), so that the inclusion $\{q^x_T>0\}\subseteq\{q^x_t>0\}$ holds (up to a $\prob$-nullset), for every $x\in E$, we obtain
\[	\ba
\expec\bigl[\xi(L)\ind_{A_t}g(L)\bigr]
&= \int_Eg(x)\expec\bigl[\xi(x)\ind_{A_t}q^x_T\bigr]\lambda(\ud x)
= \int_Eg(x)\expec\bigl[\xi(x)\ind_{A_t}q^x_T\ind_{\{q^x_t>0\}}\bigr]\lambda(\ud x)	\\
&= \int_Eg(x)\expec\left[\ind_{A_t}\expec[\xi(x)q^x_T|\F_t]\frac{q^x_t}{q^x_t}\ind_{\{q^x_t>0\}}\right]\lambda(\ud x)
= \expec\bigl[Y^{(\xi)}_t(L)\ind_{A_t}g(L)\bigr],
\ea	\]
with $Y^{(\xi)}_t(x)\dfn \expec[\xi(x)q^x_T|\F_t]\ind_{\{q^x_t>0\}}/q^x_t$, for all $x\in E$, and where we can take a c\`adl\`ag and $x$-measurable version of the $\FF$-optional projection of the random variable $\xi(x)q^x_T$ (see \cite[Proposition 3]{SY}). Since $A_t$ and $g(L)$ are arbitrary and generate the $\sigma$-field $\G^0_t$ and the random variable $Y^{(\xi)}_t(L)$ is $\G^0_t$-measurable, this shows that $\expec\bigl[\xi(L)|\G^0_t\bigr] = Y^{(\xi)}_t(L)$, for all $t\in[0,T]$.
In particular, it holds that 
\be	\label{eq:RC_proof}
\expec\bigl[\xi(L)|\G^0_{t+\varepsilon}\bigr] = Y^{(\xi)}_{t+\varepsilon}(L),
\qquad\text{ for every }\varepsilon>0.
\ee
By the martingale convergence theorem (see e.g. \cite[Corollary 2.23]{MR1219534}) and the right-continuity of $(Y^{(\xi)}_t(x))_{t\in[0,T]}$, for all $x\in E$, taking the limit for $\varepsilon\searrow0$ in \eqref{eq:RC_proof} yields $\expec\bigl[\xi(L)|\G_t\bigr] = Y^{(\xi)}_t(L)$, for all $t\in[0,T]$. The arbitrariness of $\xi$ and of $T\in(0,\infty)$ then implies that $\G_t=\G^0_t$, for all $t\in\Real_+$.
\end{proof}

In turn, Lemma \ref{lem:RC_filtr} implies a useful characterization of the optional $\sigma$-field associated to the initially enlarged filtration $\GG$, as shown in the following lemma.

\begin{lem}	\label{lem:optional}
Suppose that Assumption \ref{ass:Jacod} holds. Then for every $\GG$-optional process $Z=(Z_t)_{t\geq0}$ there exists a $\pare{\cE\otimes\cO(\FF)}$-measurable function $E\times\Omega\times\Real_+\ni(x,\omega,t)\mapsto z_t^x(\omega)$ such that $Z_t(\omega)=z_t^{L(\omega)}(\omega)$ holds $\prob$-a.s. for all $t\geq0$.
\end{lem}
\begin{proof}
If $(Z_t)_{t\geq0}\in\M(\prob,\GG)$, then $\expec[Z_n|\G_t]=Z_t$ holds for all $t\in[0,n]$, for every $n\in\N$. The proof of Lemma \ref{lem:RC_filtr} gives the existence of a $\pare{\cE\otimes\cO(\FF)}$-measurable function $Y^n_t(x)$ such that $Z_t=Y^n_t(L)$ holds $\prob$-a.s. for all $t\in[0,n]$ and for every $n\in\N$. Note that $Y^{n+1}_t(L)=Y^n_t(L)$  $\prob$-a.s. for all $t\in[0,n]$. Hence, letting $z_t^x:=\sum_{n=1}^{\infty}\ind_{[n-1,n)}(t)Y^n_t(x)$, the claim is proved for any $\GG$-martingale $Z$.
Clearly, the result also holds true for any process $(Z_t)_{t\geq0}$ of the form $Z_t=\alpha(t)M_t$, where $\alpha:\Real_+\rightarrow\Real$ and $(M_t)_{t\geq0}\in\M(\prob,\GG)$. By \cite[\textsection XX.22]{DMM}, the $\sigma$-field $\cO(\GG)$ is generated by all processes of this form, thus completing the argument.
\end{proof}

Lemma \ref{lem:optional} together with formula \eqref{expec_init} leads to a useful characterization of $\GG$-martingales in terms of a family of $\FF$-martingales parameterized by $x\in E$. The following proposition is an extension of \cite[Proposition 3.1]{CJZ} to the case where Assumption \ref{ass:Jacod} is satisfied but the equivalence $\nu_t\sim\lambda$ $\prob$-a.s. does not necessarily hold. 

\begin{prop}	\label{prop:mart}
Suppose that Assumption \ref{ass:Jacod} holds.
Then a process $(M_t)_{t\geq0}$ is a $\GG$-martingale if and only if there exists a $\pare{\cE\otimes\cO(\FF)}$-measurable function $E\times\Omega\times\Real_+\ni(x,\omega,t)\mapsto m^x_t(\omega)$ such that $(m^x_tq^x_t)_{t\geq0}\in\M(\prob,\FF)$, for all $x\in E$, and $M_t(\omega)=m_t^{L(\omega)}(\omega)$ holds $\prob$-a.s. for all $t\in\Real_+$.
\end{prop}
\begin{proof}
If $(M_t)_{t\geq0}\in\M(\prob,\GG)$, Lemma \ref{lem:optional} gives the existence of a $\pare{\cE\otimes\cO(\FF)}$-measurable function $E\times\Omega\times\Real_+\ni(x,\omega,t)\mapsto \widetilde{m}_t^x(\omega)$ such that $M_t(\omega)=\widetilde{m}_t^{L(\omega)}(\omega)$ holds $\prob$-a.s. for all $t\in\Real_+$.
Let $s,t\in\mathbb{Q}_+$ (with $\mathbb{Q}_+$ denoting the set of positive rational numbers) with $s\leq t$ and consider an arbitrary event $A_s\in\F_s$ and a bounded $\cE$-measurable function $g:E\rightarrow\Real$. Then, in view of formula \eqref{expec_init} together with the $\GG$-martingale property of $(M_t)_{t\geq0}$, it holds that
\begin{align}
\int_Eg(x)\expec[\widetilde{m}^x_sq^x_s\ind_{A_s}]\lambda(\ud x)
&= \expec[\widetilde{m}^L_s\ind_{A_s}g(L)]
= \expec[M_s\ind_{A_s}g(L)]	\notag\\
&= \expec[M_t\ind_{A_s}g(L)]
= \expec[\widetilde{m}^L_t\ind_{A_s}g(L)]
= \int_Eg(x)\expec\Bigl[\expec[\widetilde{m}^x_tq^x_t|\F_s]\ind_{A_s}\Bigr]\lambda(\ud x),
\label{eq:mart_proof}	
\end{align}
where in the last term we take a $\cE$-measurable version of the conditional expectation $\expec[\widetilde{m}^x_tq^x_t|\F_s]$ (see \cite[Lemma 3]{SY}). This shows that the set 
$
\bigl\{(x,\omega)\in E\times\Omega : \widetilde{m}^x_s(\omega)q^x_s(\omega)\neq\expec[\widetilde{m}^x_tq^x_t|\F_s](\omega)\bigr\}
$
has zero $(\lambda\otimes\PP)$-measure. In turn, arguing similarly as in \cite[Lemma 1.8]{Jac85}, this implies that the set
$
B := \{x\in E : \widetilde{m}^x_sq^x_s = \expec[\widetilde{m}^x_tq^x_t|\F_s]\text{ $\prob$-a.s. for all }s,t\in\mathbb{Q}_+\text{ with }s\leq t\}
$
satisfies $\lambda(B)=1$.
Define then $\widehat{m}^x_t(\omega):=\widetilde{m}^x_t(\omega)\ind_B(x)$, for all $(x,\omega,t)\in E\times\Omega\times\Real_+$, so that $(\widehat{m}^x_tq^x_t)_{t\in\mathbb{Q}_+}$ is a martingale on $\basisp$, for all $x\in E$.
For every $t\in\mathbb{Q}_+$, it holds that
\[	\ba
\prob(M_t=\widehat{m}^L_t)
&= \prob(\widetilde{m}^L_t=\widehat{m}^L_t)
= \int_E \expec\bigl[\ind_{\{\widetilde{m}^x_t=\widehat{m}^x_t\}}q^x_t\bigr]\lambda(\ud x)
= \int_B \expec[q^x_t]\lambda(\ud x)
= \lambda(B) = 1,
\ea	\]
where the first equality uses the fact that $\prob(M_t=\widetilde{m}^L_t)=1$  and the second equality follows from \eqref{expec_init}. Hence, $M_t=\widehat{m}^L_t$ holds $\prob$-a.s. simultaneously for all rationals $t\in\mathbb{Q}_+$.
For $t\geq0$, following the proof of \cite[Lemma 1.8]{Jac85}, let denote by $C_t$ the set of all $(x,\omega)\in E\times\Omega$ such that $\widehat{m}^x_{\cdot}(\omega)q^x_{\cdot}(\omega)$ admits finite limits from the left and from the right along the rationals at every $s\in[0,t]$. 
For every $t>0$, the set $C_t$ is $\F_t$-measurable and the martingale property of $(\widehat{m}^x_tq^x_t)_{t\in\mathbb{Q}_+}$ implies that $\prob(\{\omega:(x,\omega)\in C_t\})=1$, for every $x\in E$.
For all $(x,\omega,t)\in E\times\Omega\times\R_+$, define then
\[
n^x_t(\omega) :=
\begin{cases}
\underset{s\in\mathbb{Q}_+,s\downarrow t}{\lim} \, \widehat{m}^x_s(\omega)q^x_s(\omega),
& \text{ if }(x,\omega)\in\bigcap_{s>t}C_s;\\
0, & \text{ otherwise}.
\end{cases}
\]
The function $E\times\Omega\times\R_+\ni(x,\omega,t)\mapsto n^x_t(\omega)$ is $\cO(\oFF)$-measurable and, since $n^x_{\cdot}(\omega)$ is c\`adl\`ag for every $(x,\omega)\in E\times\Omega$, it is also $\pare{\cE\otimes\cO(\FF)}$-measurable, in view of \cite[Remark 1 after Proposition 3]{SY}. 
Hence, for every $x\in E$, $(n^x_t)_{t\geq0}$ is a right-continuous regularization of the martingale $(\widehat{m}^x_tq^x_t)_{t\in\mathbb{Q}_+}$.
Recalling that $q^x>0$ on $\dbraco{0,\zeta^x}$ (see e.g. \cite[Theorem 2.62]{MR1219534}), let then $m^x_t(\omega):=n^x_t(\omega)/q^x_t(\omega)\ind_{\{t<\zeta^x(\omega)\}}$, for all $(x,\omega,t)\in E\times\Omega\times\R_+$. 
Since $q^x=0$ on $\dbraco{\zeta^x,\infty}$, it holds that $m^x_tq^x_t=n^x_t\ind_{\{t<\zeta^x\}}=n^x_t$, so that $(m^x_tq^x_t)_{t\geq0}$ also represents a right-continuous regularization of $(\widehat{m}^x_tq^x_t)_{t\in\mathbb{Q}_+}$.
In view of Lemma \ref{lem-Jac}, the function $(x,\omega,t)\mapsto m^x_t(\omega)$ is $\pare{\cE\otimes\cO(\FF)}$-measurable and $m^x_{\cdot}(\omega)$ is c\`adl\`ag, for every $(x,\omega)\in E\times\Omega$.
For any $t\geq0$, let $\{t_n\}_{n\in\N}$ be a decreasing sequence in $\mathbb{Q}_+$ such that $\lim_{n\rightarrow\infty}t_n=t$. 
It then holds that
\[
m^L_t = \frac{n^L_t}{q^L_t} = \lim_{n\rightarrow\infty}\widehat{m}^L_{t_n} = \lim_{n\rightarrow\infty}M_{t_n} = M_t
\qquad\text{ $\prob$-a.s.}, 
\]
where we have used the fact that $M_{\cdot}=\widehat{m}^L_{\cdot}$ holds on all positive rationals $\prob$-a.s. together with the right-continuity of $(M_t)_{t\geq0}$. This proves the necessity part of the proposition.


The converse implication (sufficiency) follows by the same arguments used in  \eqref{eq:mart_proof} together with the fact that the $\sigma$-field $\G_s$ is generated by the random variables of the form $\ind_{A_s}g(L)$, for $A_s\in\F_s$ and $g:E\rightarrow\Real$ bounded $\cE$-measurable (see Lemma \ref{lem:RC_filtr}). 
\end{proof}

Proposition \ref{prop:mart} characterizes the $\GG$-martingale property by separating the dependence on the original filtration $\FF$ from the additional information generated by $L$, making use of the densities $\{q^x:x\in E\}$. 
Similar results have already appeared in the theory of enlargement of filtrations.
In particular, the result of Proposition \ref{prop:mart} is in the spirit of the Girsanov-like interpretation of initial enlargement proposed in \cite[Section 5]{Jac85}.
Furthermore, noting that an initial enlargement of $\FF$ with respect to $L$ coincides with a progressive enlargement of $\FF$ with respect to $L$ on $\dbraco{L,\infty}$ (see e.g. \cite[Lemma 3]{KLP}), an analogous result has been established in \cite[Proposition 5.5]{EKJJ}. A related result can also be found in  \cite{GVV06}.

\subsection{The predictable representation property on the product space $(\widehat{\Omega},\oFF,\oPP)$}	\label{subsec2}

The main purpose of this subsection consists in showing that the strong predictable representation property of the $\FF$-local martingale $S=(S_t)_{t\geq0}$ can be transferred onto the product space $(\widehat{\Omega},\oFF,\oPP)$. In particular, if $\{(m^x_t)_{t\geq0} : x\in E\}$ is a family of $\FF$-martingales parameterized by $x\in E$, we shall establish a martingale representation result which holds simultaneously for all $(m^x_t)_{t\geq0}$, $x\in E$, with good measurability properties of the integrand appearing in the representation (see Proposition \ref{prop:mr} and compare also with Remark \ref{rem:Jacod} and Appendix \ref{sec:alt}).
The results presented in this subsection do not rely on Assumption \ref{ass:Jacod}.
 
In the line of \cite{FI93}, working on the product space $(\widehat{\Omega},\oFF,\oPP)$ allows to decouple the random variable $L$ from the original filtration $\FF$. Hence, by embedding the original probability space into the product space, the situation becomes analogous to an initial enlargement of $\FF$ with an independent random variable and Lemmata \ref{lem:EI_mart} and \ref{lem:PRP_prod} could therefore be deduced from known results on initially enlarged filtrations. However, in order to make the presentation self-contained, we prefer to give complete proofs for all the results.

Let $E\times\Omega\times\Real_+\ni(x,\omega,t)\mapsto \mu^x_t(\omega)$ be an $(\oA\otimes\mathcal{B}_{\Real_+})$-measurable function. We use the notation $\mu^{\cdot}=(\mu^{\cdot}_t)_{t\geq0}$ to denote the map $(x,\omega,t)\mapsto \mu^x_t(\omega)$ viewed as a stochastic process on the product space $(\widehat{\Omega},\oFF,\oPP)$,
As a preliminary, we show that the martingale property on the product space $(\widehat{\Omega},\oFF,\oPP)$ can be characterized in terms of the martingale property of a family of processes on the original space $\basisp$. The sufficiency part of the following lemma has been recently established in \cite[Proposition 4.7]{EsmImk}.
Even though the proof of the following lemma is rather similar to that of Proposition \ref{prop:mart}, we include full details for the convenience of the reader.

\begin{lem}	\label{lem:EI_mart}
An $(\oA\otimes\mathcal{B}_{\Real_+})$-measurable function $E\times\Omega\times\Real_+\ni(x,\omega,t)\mapsto \mu^x_t(\omega)$ which satisfies $\int_E\expec[|\mu_t^x|]\lambda(\ud x)<\infty$, for all $t\in\Real_+$, is a martingale on $(\widehat{\Omega},\oFF,\oPP)$ if and only if there exists a $\pare{\cE\otimes\cO(\FF)}$-measurable function $E\times\Omega\times\Real_+\ni (x,\omega,t)\mapsto m^x_t(\omega)$ such that $(m^x_t)_{t\geq0}\in\M(\prob,\FF)$, for all $x\in E$, and $\mu^x_t(\omega)=m^x_t(\omega)$ holds $\oPP$-a.s. for all $t\in\Real_+$.
\end{lem}
\begin{proof}
Let $(x,\omega,t)\mapsto \mu^x_t(\omega)$ be a measurable function satisfying $\int_E\expec[|\mu_t^x|]\lambda(\ud x)<\infty$, for all $t\in\R_+$, and such that $(\mu^{\cdot}_t)_{t\geq0}\in\M(\oPP,\oFF)$. 
Being right-continuous in $t$ and $\oFF$-adapted, the map $(x,\omega,t)\mapsto\mu^x_t(\omega)$ is $\cO(\oFF)$-measurable, so that the process $(\mu^x_t)_{t\geq0}$ is $\cO(\FF)$-measurable, for every $x\in E$ (see Section \ref{subsec:setting}). Without loss of generality, we can take a $\pare{\cE\otimes\cO(\FF)}$-measurable version of the map $(x,\omega,t)\mapsto \mu^x_t(\omega)$, see \cite[Remark 1 after Proposition 3]{SY}. Consider arbitrary $s,t\in\mathbb{Q}_+$ with $s\leq t$, $A_s\in\F_s$ and a bounded $\cE$-measurable function $g:E\rightarrow\Real$. Then the random variable $g(\cdot)\ind_{A_s}$ on $(\widehat{\Omega},\oA)$ is $\oF_s$-measurable and, hence, the martingale property of $(\mu^{\cdot}_t)_{t\geq0}$ implies that
\[
\int_Eg(x)\expec\bigl[\mu^x_s\ind_{A_s}\bigr]\lambda(\ud x)
= \oE[\mu^{\cdot}_s\,g(\cdot)\ind_{A_s}]
= \oE[\mu^{\cdot}_t\,g(\cdot)\ind_{A_s}]
= \int_Eg(x)\expec\bigl[\expec[\mu^x_t|\F_s]\ind_{A_s}\bigr]\lambda(\ud x),
\]
where we take a $\cE$-measurable version of the conditional expectation $\expec[\mu^x_t|\F_s]$. 
By the same arguments used in the proof of Proposition \ref{prop:mart}, we then get the existence of a set $B\in\cE$ such that $\lambda(B)=1$ and of a $\pare{\cE\otimes\cO(\FF)}$-measurable function $E\times\Omega\times\Real_+\ni(x,\omega,t)\mapsto\widehat{m}^x_t(\omega)$ such that $(\widehat{m}^x_t)_{t\in\mathbb{Q}_+}$ is a martingale on $\basisp$, for all $x\in E$, and $\widehat{m}^x_t(\omega)=\mu^x_t(\omega)$ holds $\prob$-a.s. for all $t\in\mathbb{Q}_+$ and $x\in B$. 
Moreover, for every $t\in\mathbb{Q}_+$, it holds that
$
\oPP(\mu^{\cdot}_t=\widehat{m}^{\cdot}_t)
= \int_E\prob(\mu^x_t=\widehat{m}^x_t)\lambda(\ud x)
= \lambda(B) = 1.
$
Define then $(m^x_t)_{t\geq0}$ as the right-continuous regularization of $(\widehat{m}^x_t)_{t\in\mathbb{Q}_+}$, for every $x\in E$, as constructed in the proof of Proposition \ref{prop:mart}, so that $(m^x_t)_{t\geq0}\in\M(\prob,\FF)$, for every $x\in E$. 
The necessity part of the proposition then follows by the right-continuity of $(\mu^{\cdot}_t)_{t\geq0}$.

The converse implication (sufficiency) can be proved exactly as in \cite[Proposition 4.7]{EsmImk}.
\end{proof}

Let us define the $\oFF$-adapted process $\Shat=(\Shat_t)_{t\geq0}$ by $\Shat_t(x,\omega):=S_t(\omega)$, for all $(x,\omega,t)\in\widehat{\Omega}\times\Real_+$. 
As shown in the following (trivial) result, the process $\Shat$ on $(\widehat{\Omega},\oFF,\oPP)$ inherits the local martingale property of the original process $S$ on $\basisp$.

\begin{cor}	\label{lem:Shat_mart}
The process $\Shat=(\Shat_t)_{t\geq0}$ is a local martingale on $(\widehat{\Omega},\oFF,\oPP)$.
\end{cor}
\begin{proof}
Since $(S_t)_{t\geq0}\in\Mloc(\prob,\FF)$, there exists a sequence $\{\tau_n\}_{n\in\N}$ of $\FF$-stopping times increasing $\prob$-a.s. to infinity such that $S^{\tau_n}\in\M(\prob,\FF)$, for all $n\in\N$. Letting $\widehat{\tau}_n(x,\omega):=\tau_n(\omega)$, for all $(x,\omega)\in\widehat{\Omega}$, it  holds that 
$ 
\{(x,\omega)\in\widehat{\Omega} : \widehat{\tau}_n(x,\omega)\leq t\}
= E \times \{\omega\in\Omega : \tau_n(\omega)\leq t\}
\in \oF_t
$,
for all $t\in\Real_+$ and $n\in\N$, so that $\{\widehat{\tau}_n\}_{n\in\N}$ are $\oFF$-stopping times. Since $\Shat_{\widehat{\tau}_n\wedge t}=S_{\tau_n\wedge t}$ and $(S_{\tau_n\wedge t})_{t\geq0}\in\M(\prob,\FF)$, Lemma \ref{lem:EI_mart} implies that $(\Shat_{\widehat{\tau}_n\wedge t})_{t\geq0}\in\M(\oPP,\oFF)$, for all $n\in\N$, thus proving the claim.
\end{proof}

We are now in a position to prove that the predictable representation property of  $S$ on $\basisp$ can be transferred onto the product space $(\widehat{\Omega},\oFF,\oPP)$. 
This is the content of Lemma \ref{lem:PRP_prod} below, which can be regarded as an extension of \cite[Theorem 4.13]{EsmImk} to a general setting.
As a preliminary, we recall the following well-known characterization of the strong predictable representation property (see \cite[Corollary 4.12]{MR542115} and also \cite[Theorem 13.5]{MR1219534} for the one-dimensional case), formulated on a generic filtered probability space $(\Omega',\A',\FF',\prob')$.

\begin{prop}	\label{prop:MRP}
Let $X=(X_t)_{t\geq0}$ be an $\Real^d$-valued local martingale on $(\Omega',\FF',\prob')$. \\
The following are equivalent:
\begin{enumerate}
\item[(i)] $X$ has the strong predictable representation property on $(\Omega',\FF',\prob')$;
\item[(ii)] for every bounded $N\in\M(\prob',\FF')$ with $N_0=0$, if $NX^i\in\mMloc(\prob',\FF')$ for all $i=1,\ldots,d$, then $N=0$ (up to a $\prob'$-evanescent set).
\end{enumerate}
\end{prop}

Lemma \ref{lem:EI_mart}, Corollary \ref{lem:Shat_mart} and Proposition \ref{prop:MRP} then yield the following result.

\begin{lem}	\label{lem:PRP_prod}
Suppose that $S=(S_t)_{t\geq0}$ has the strong predictable representation property on $(\Omega,\FF,\prob)$.
Then $\Shat=(\Shat_t)_{t\geq0}$ has the strong predictable representation property on $(\widehat{\Omega},\oFF,\oPP)$.
\end{lem}
\begin{proof}
Let $\widehat{N}=(\widehat{N}_t)_{t\geq0}$ be a bounded martingale on $(\widehat{\Omega},\oFF,\oPP)$ with $\widehat{N}_0=0$ and suppose that $\widehat{N}\Shat^i\in\Mloc(\oPP,\oFF)$, for all $i=1,\ldots,d$. It is easy to check that $(\widehat{N}\Shat^i)^{\widehat{\tau}_n}\in\M(\oPP,\oFF)$, for all $i=1,\ldots,d$ and $n\in\N$, where the $\oFF$-stopping times $\{\widehat{\tau}_n\}_{n\in\N}$ are as in the proof of Corollary \ref{lem:Shat_mart}.
By the same arguments used in the proof of Lemma \ref{lem:EI_mart}, there exists a $\pare{\cE\otimes\cO(\FF)}$-measurable function $E\times\Omega\times\Real_+\ni(x,\omega,t)\mapsto N^x_t(\omega)$ such that $N^x_t(\omega)=\widehat{N}^x_t(\omega)$ holds $\oPP$-a.s. for all $t\in\Real_+$ and satisfying $(N^x_t)_{t\geq0}\in\M(\prob,\FF)$ and $(N^x_{\tau_n\wedge t}S^i_{\tau_n\wedge t})_{t\geq0}\in\M(\prob,\FF)$, for all $x\in E$ and every $i=1,\ldots,d$ and $n\in\N$, where the $\FF$-stopping times $\{\tau_n\}_{n\in\N}$ are as in the proof of Corollary \ref{lem:Shat_mart}. In turn, this means that $(N^x_tS^i_t)_{t\geq0}\in\Mloc(\prob,\FF)$, for all $x\in E$. By Proposition \ref{prop:MRP}, the strong predictable representation property of $S$ on $(\Omega,\FF,\prob)$ implies that $N^x=0$ $\prob$-a.s., for all $x\in E$. Since $\widehat{N}^x_t(\omega)=N^x_t(\omega)$ holds $\oPP$-a.s. for all $t\in\Real_+$, this implies that $\widehat{N}=0$ up to a $\oPP$-evanescent set. Again by Proposition \ref{prop:MRP}, this proves that $\Shat$ has the strong predictable representation property on $(\widehat{\Omega},\oFF,\oPP)$.
\end{proof}

In particular, the above lemma allows us to prove the following martingale representation result, which holds simultaneously for a family  $\{(m^x_t)_{t\geq0} : x\in E\}$ of measurable processes such that $(m^x_t)_{t\geq0}\in\M(\prob,\FF)$, for all $x\in E$. This represents the key result of the present subsection.

\begin{prop}	\label{prop:mr}
Suppose that $S=(S_t)_{t\geq0}$ has the strong predictable representation property on $(\Omega,\FF,\prob)$. 
Let $E\times\Omega\times\Real_+\ni(x,\omega,t)\mapsto m^x_t(\omega)$ be a $\pare{\cE\otimes\cO(\FF)}$-measurable function satisfying $\int_E\expec[|m^x_t|]\lambda(\ud x)<\infty$, for all $t\in\Real_+$, and such that $(m^x_t)_{t\geq0}\in\M(\prob,\FF)$, for all $x\in E$.
Then there exists a $\pare{\cE\otimes\cP(\FF)}$-measurable function $E\times\Omega\times\Real_+\ni(x,\omega,t)\mapsto \theta^x_t(\omega)\in\Real^d$ satisfying $\theta^x\in L_m(S;\prob,\FF)$, for all $x\in E$, such that $m^x_t(\omega)=m^x_0(\omega)+(\theta^x\cdot S)_t(\omega)$ holds $\oPP$-a.s. for all $t\in\R_+$.
\end{prop}
\begin{proof}
By Lemma \ref{lem:EI_mart}, it holds that $(m^{\cdot}_t)_{t\geq0}\in\M(\oPP,\oFF)$. Lemma \ref{lem:PRP_prod} gives then the existence of a $\cP(\oFF)$-measurable function $\widehat{\Omega}\times\Real_+\ni(x,\omega,t)\mapsto\widetilde{\theta}^x_t(\omega)\in\Real^d$ satisfying $\widetilde{\theta}^{\cdot}\in L_m(\Shat;\oPP,\oFF)$ such that $m^{\cdot}_t=m^{\cdot}_0+(\widetilde{\theta}^{\cdot}\cdot S)_t$ holds $\oPP$-a.s. for all $t\in\R_+$.
Since $\cP(\oFF)=\cE\otimes\cP(\FF)$, the map $(x,\omega,t)\mapsto\widetilde{\theta}^x_t(\omega)$ is $\pare{\cE\otimes\cP(\FF)}$-measurable. Moreover, noting that $[\Shat,\Shat](x,\omega)=[S,S](\omega)$, for all $(x,\omega)\in\widehat{\Omega}$, the fact that $\widetilde{\theta}^{\cdot}\in L_m(\Shat;\oPP,\oFF)$ together with \cite[\textsection (4.59)]{MR542115} can be easily shown to imply that $\widetilde{\theta}^x\in L_m(S;\prob,\FF)$, for all $x$ belonging to a  set $B\in\cE$ with $\lambda(B)=1$. Define then $\theta^x_t(\omega):=\widetilde{\theta}^x_t(\omega)\ind_B(x)$, for all $(x,\omega,t)\in E\times\Omega\times\Real_+$, so that $\theta^x\in L_m(S;\prob,\FF)$, for all $x\in E$. For all $t\in\R_+$, it holds that
\[
\oPP\bigl(m^{\cdot}_t-m^{\cdot}_0=(\theta^{\cdot}\cdot S)_t\bigr)
= \int_E\expec\bigl[\ind_{\{m^x_t-m^x_0=(\theta^x\cdot S)_t\}}\bigr]\lambda(\ud x)
= \int_B\expec\bigl[\ind_{\{m^x_t-m^x_0=(\widetilde{\theta}^x\cdot S)_t\}}\bigr]\lambda(\ud x) 
= \lambda(B) = 1,
\]
where we take a version of the stochastic integral which is measurable in $x$, which exists by \cite[Theorem 2]{SY}.\footnote{Note that, in view of the note on page 133 of \cite{SY}, the assumption that the space $L^1(\Omega,\A,\prob)$ is separable is not needed in the proof of \cite[Theorem 2]{SY}.} 
\end{proof}

In particular, if Assumption \ref{ass:Jacod} holds, Lemma \ref{lem-Jac} shows that the function $(x,\omega,t)\mapsto q^x_t(\omega)$ satisfies the hypotheses of Proposition \ref{prop:mr}. Hence, there exists a $\pare{\cE\otimes\cP(\FF)}$-measurable function $E\times\Omega\times\Real_+\ni(x,\omega,t)\mapsto H^x_t(\omega)\in\Real^d$ satisfying $H^x\in L_m(S;\prob,\FF)$, for all $x\in E$, and such that 
\be	\label{eq:repr_q}
q^x_t(\omega) = q^x_0(\omega) + (H^x\cdot S)_t(\omega)
\qquad\text{ $\oPP$-a.s. for all $t\in\R_+$.}
\ee

\subsection{An auxiliary representation result}	\label{sec:aux_repr}

In this subsection, we combine the results obtained in the two preceding subsections and prove an auxiliary representation result that will turn out to be a key step in the derivation of the martingale representations stated in section \ref{sec:main_results}. 
The following result relies on Propositions \ref{prop:mart} and \ref{prop:mr} and shows that every local martingale on $\basisgp$ can be represented as a stochastic integral of $S$ up to a suitable ``change of num\'eraire''.

\begin{prop}	\label{lem:repr_1}
Suppose that Assumption \ref{ass:Jacod} holds and that the process $S=(S_t)_{t\geq0}$ has the strong predictable representation property on $(\Omega,\FF,\prob)$. 
Let $M=(M_t)_{t\geq0}\in\M(\prob,\GG)$.
Then there exists a $\pare{\cE\otimes\cP(\FF)}$-measurable function $E\times\Omega\times\Real_+\ni(x,\omega,t)\mapsto K^x_t(\omega)\in\Real^d$ satisfying $K^x\in L_m(S;\prob,\FF)$, for all $x\in E$,  and such that
\be	\label{eq:repr_1}
M_t = \frac{1}{q^L_t}\bigl(q^L_0M_0 + (K^L\cdot S)_t\bigr)
\qquad\text{ $\prob$-a.s. for all $t\in\R_+$,}
\ee
with the stochastic integral being understood as a semimartingale stochastic integral in $\GG$.
\end{prop}
\begin{proof}
If $M\in\M(\prob,\GG)$, Proposition \ref{prop:mart} gives the existence of a $\pare{\cE\otimes\cO(\FF)}$-measurable function $(x,\omega,t)\mapsto m^x_t(\omega)$ such that $(q^x_tm^x_t)_{t\geq0}\in\M(\prob,\FF)$, for all $x\in E$, and such that $M_t=m^L_t$ holds $\prob$-a.s. for all $t\in\R_+$.
Since the map $(x,\omega,t)\mapsto q^x_t(\omega)$ is also $\pare{\cE\otimes\cO(\FF)}$-measurable (see Lemma \ref{lem-Jac}) and $\int_E\expec[|q^x_tm^x_t|]\lambda(\ud x)=\expec[|M_t|]<\infty$, for all $t\in\R_+$, Proposition \ref{prop:mr} implies that there exists an $\Real^d$-valued $\pare{\cE\otimes\cP(\FF)}$-measurable function $(x,\omega,t)\mapsto K^x_t(\omega)$ satisfying $K^x\in L_m(S;\prob,\FF)$, for all $x\in E$, such that 
\be	\label{eq:prod_repr}
q^x_t(\omega)m^x_t(\omega) = q^x_0(\omega)m^x_0(\omega) + (K^x\cdot S)_t(\omega) 
\qquad\text{  $\oPP$-a.s. for all $t\in\R_+$.}
\ee
By \cite[Theorem 2]{SY}, there exists a $\pare{\cE\otimes\cO(\FF)}$-measurable version of the stochastic integral $K^x\cdot S$.
Moreover, since every $\FF$-semimartingale is a $\GG$-semimartingale (see \cite[Theorem 1.1]{Jac85}), \cite[Proposition 2.1]{MR604176} implies that $K^x\cdot S$ is well-defined as a semimartingale stochastic integral in $\GG$, for every $x\in E$ (i.e., $K^x\in L(S;\prob,\GG)$ for every $x\in E$). Furthermore, in view of \cite[Proposition III.6.25]{MR1943877}, the stochastic integral $K^x\cdot S$ is the same when considered with respect to either of the two filtrations $\FF$ and $\GG$, so that there also exists an $x$-measurable version of $K^x\cdot S$ when considered in the filtration $\GG$.
Assuming for the moment that the $\GG$-predictable process $K^L=(K^L_t)_{t\geq0}$ belongs to $L(S;\prob,\GG)$, the stochastic integral $K^L\cdot S$ is well-defined in $\GG$.
Moreover, $(K^x\cdot S)|_{x=L}$ is indistinguishable from $K^L\cdot S$. This is evident if $K^x=k(x)K$, for some $\cE$-measurable bounded function $k:E\rightarrow\R$ and some $\bF$-predictable bounded process $K=(K_t)_{t\geq0}$, while the general case follows from a monotone class argument, using \cite[Proposition 5]{SY} together with the dominated convergence theorem for stochastic integrals (see \cite[Theorem IV.32]{MR1037262}).
Recalling that $q^L_t>0$ $\prob$-a.s.,
we can therefore conclude that, for all $t\in\R_+$,
\[
M_t = m^L_t 
= \frac{1}{q^L_t}(q^x_tm^x_t)\bigr|_{x=L}
= \frac{1}{q^L_t}\bigl(q^x_0m^x_0 + (K^x\cdot S)_t\bigr)\Bigr|_{x=L}
= \frac{1}{q^L_t}\bigl(q^L_0M_0 + (K^L\cdot S)_t\bigr)
\quad\text{ $\prob$-a.s.}
\]

To complete the proof, it remains to prove that $K^L\in L(S;\prob,\GG)$. To this effect, let $S^c$ and $S^d$ denote the continuous and purely discontinuous, respectively, $\FF$-local martingale parts of $S$ (see \cite[Theorem I.4.18]{MR1943877}). 
By \cite[Theorem 1.1]{Jac85}, $S^c$ and $S^d$ are special semimartingales in  $\GG$.
Hence, let $S^c=\Stilde^{(c)}+A^{(c)}$ and $S^d=\Stilde^{(d)}+A^{(d)}$ denote the $\GG$-canonical decompositions of $S^c$ and $S^d$, respectively, where $\Stilde^{(c)},\Stilde^{(d)}\in\Mloc(\prob,\GG)$ and $A^{(c)},A^{(d)}$ are $\GG$-predictable processes of finite variation. Therefore, we can write
\[
S = S_0 + S^c + S^d
= S_0 + \Stilde^{(c)} + \Stilde^{(d)} + A^{(c)} + A^{(d)}.
\]
In the remaining part of the proof, which relies on \cite{Jac85}, we shall prove that
\[
K^L \in L_m(\Stilde^{(c)};\prob,\GG) \cap L_m(\Stilde^{(d)};\prob,\GG) \cap L^0(A^{(c)};\prob,\GG) \cap L^0(A^{(d)};\prob,\GG) \subseteq L(S;\prob,\GG),
\]
where $L^0(A^{(c)};\prob,\GG)$ denotes the space of all $\R^d$-valued $\GG$-predictable processes which are integrable (in the sense of \cite[III.\textsection6b]{MR1943877}) with respect to the finite variation process $A^{(c)}$, and similarly for $L^0(A^{(d)};\prob,\GG)$.\\
(i):
by assumption, for every $x\in E$, it holds that $K^x\in L_m(S;\prob,\FF) \subseteq L_m(S^c;\prob,\FF)$. Therefore, $\int_0^t(K^x_s)^{\top}\ud\langle S^c,S^c\rangle^{\FF}_sK^x_s<\infty$ $\prob$-a.s. for all $t\in\R_+$ and $x\in E$. In view of \cite[Remark 9.20]{MR542115}, it holds that $\langle\Stilde^{(c)},\Stilde^{(c)}\rangle^{\GG}=\langle S^c,S^c\rangle^{\FF}$. An application of formula \eqref{expec_init} then gives, for all $t\in\R_+$,
\begin{align*}
\prob\left(\int_0^t(K^L_s)^{\top}\ud\langle\Stilde^{(c)},\Stilde^{(c)}\rangle^{\GG}_sK^L_s<\infty\right)
&= \prob\left(\int_0^t(K^L_s)^{\top}\ud\langle S^c,S^c\rangle^{\FF}_sK^L_s<\infty\right)	\\
&= \int_E\expec\left[q^x_t\ind_{\{\int_0^t(K^x_s)^{\top}\ud\langle S^c,S^c\rangle^{\FF}_sK^x_s<\infty\}}\right]\lambda(\ud x)
= \int_E\expec[q^x_t]\lambda(\ud x)=1,
\end{align*}
thus proving that $K^L\in L_m(\Stilde^{(c)};\prob,\GG)$.\\
(ii): 
by \cite[Theorem 2.1]{Jac85}, it holds that $A^{(c)}=\int_0^{\cdot}\frac{1}{q^L_{s-}}\ud\langle S^c,q^x\rangle^{\FF}_s\bigr|_{x=L}=\int_0^{\cdot}\frac{1}{q^L_{s-}}\ud\langle S^c,S^c\rangle^{\FF}_sH^L_s$, where the second equality follows from \eqref{eq:repr_q}. 
The same arguments used in step (i) allow to show that the $\GG$-predictable process $H^L=(H^L_t)_{t\geq0}$ satisfies $\int_0^t(H^L_s)^{\top}\ud\langle S^c,S^c\rangle^{\FF}_sH^L_s<\infty$ $\prob$-a.s. for all $t\in\R_+$. Hence, the fact that $\int_0^t(K^L_s)^{\top}\ud\langle S^c,S^c\rangle^{\FF}_sK^L_s<\infty$ $\prob$-a.s. for all $t\in\R_+$ (see step (i)) and the Kunita-Watanabe inequality imply that $\int_0^t|(K^L_s)^{\top}\ud\langle S^c,S^c\rangle^{\FF}_sH^L_s|<\infty$ $\prob$-a.s. for all $t\in\R_+$. Since the process $1/q^L_-$ is locally bounded, this proves that $K^L\in L^0(A^{(c)};\prob,\GG)$.\\
(iii): 
let $\mu^S(\omega;\ud t,\ud y)$ denote the jump measure of $S$, in the sense of \cite[Proposition II.1.16]{MR1943877}, and $\nu^{S,\FF}(\omega;\ud t,\ud y)$ the corresponding compensating measure in $\FF$. 
If $W:\Omega\times\Real_+\times\Real^d\ni(\omega,t,y)\mapsto W(\omega,t,y)$ is a $\pare{\cP(\FF)\otimes\mathcal{B}_{\Real^d}}$-measurable function, we denote by $W\ast(\mu^S-\nu^{S,\FF})$ the stochastic integral (when it exists) with respect to the random measure $\mu^S-\nu^{S,\FF}$, in the sense of \cite[Definition II.1.27]{MR1943877}. We denote by $\mathcal{G}_{\text{loc}}(\mu^S;\FF)$ the set of all  $\pare{\cP(\FF)\otimes\mathcal{B}_{\Real^d}}$-measurable functions $W$ such that the stochastic integral $W\ast(\mu^S-\nu^{S,\FF})$ exists.
In view of \cite[Corollary II.2.38]{MR1943877}, it holds that $S^d=y\ast(\mu^S-\nu^{S,\FF})$, where $y$ denotes the map $(\omega,t,y)\mapsto y\in\R^d$. 
Since $S$ has the strong predictable representation property on $\basisp$, the purely discontinuous $\FF$-local martingale part $q^{x,d}$ of $q^x$ admits a representation of the form $q^{x,d}=(q^x_-U^x)\ast(\mu^S-\nu^{S,\FF})$, for every $x\in E$, where the map $(x,\omega,t,y)\mapsto U^x(\omega,t,y)$ is $\pare{\cE\otimes\cP(\FF)\otimes\cB_{\R^d}}$-measurable and can be chosen to satisfy properties (ii)-(iii)-(iv) of \cite[Proposition 3.14]{Jac85}.
Moreover, \cite[Theorem 4.1]{Jac85} shows that the compensator $\nu^{S,\GG}$ of $\mu^S$ in the filtration $\GG$ is given by
\be	\label{eq:compensator_G}
\nu^{S,\GG}(\omega;\ud t,\ud y) = \bigl(1+U^{L(\omega)}(\omega,t,y)\bigr)\nu^{S,\FF}(\omega;\ud t,\ud y).
\ee
By \eqref{eq:prod_repr}, it holds that $(K^x)^{\top}\Delta S=\Delta (q^xm^x)$ $\oPP$-a.s., so that $(K^L)^{\top}\Delta S=\Delta (q^Lm^L)$ $\prob$-a.s.
Since $(q^x_tm^x_t)_{t\geq0}\in\M(\prob,\FF)$, for all $x\in E$, \cite[Proposition 4]{ACJ15} implies that $(q^L_tm^L_t)_{t\geq0}$ is a special semimartingale in $\GG$. In turn, by \cite[Proposition II.2.29]{MR1943877}, this implies that the process $\sum_{0<s\leq\cdot}((K^L_s)^{\top}\Delta S_s)^2\wedge|(K^L_s)^{\top}\Delta S_s|=\sum_{0<s\leq\cdot}(\Delta(q^Lm^L)_s)^2\wedge|\Delta(q^Lm^L)_s|$ is $\GG$-locally integrable.
Hence, for all $t\in\R_+$, it holds that,
\be	\label{eq:drift_integr_1}	\begin{aligned}
&\int_0^t\int_{\R^d}\Bigl(\bigl((K^L_s)^{\top}y\bigr)^2\wedge\bigl|(K^L_s)^{\top}y\bigr|\Bigr)\nu^{S,\GG}(\ud s,\ud y)	\\
&\quad = \int_0^t\int_{\R^d}\Bigl(\bigl((K^L_s)^{\top}y\bigr)^2\wedge\bigl|(K^L_s)^{\top}y\bigr|\Bigr)\bigl(1+U^{L}(s,y)\bigr)\nu^{S,\FF}(\ud s,\ud y) < \infty
\qquad \text{ $\prob$-a.s.}
\end{aligned}	\ee
where we have used \eqref{eq:compensator_G}.
Moreover, since $K^x\in L_m(S;\prob,\FF)\subseteq L_m(S^d;\prob,\FF)$, for every $x\in E$, and $S^d=y\ast(\mu^S-\nu^{S,\FF})$, it holds that $(K^x)^{\top}y\in\mathcal{G}_{\text{loc}}(\mu^S;\FF)$, for every $x\in E$ (see e.g. \cite[Theorem 11.23]{MR1219534}). 
In turn, noting that $\int_{\R^d}y\nu^{S,\FF}(\{t\}\times\ud y)=0$ up to an evanescent set and in view of \cite[Theorem II.1.33]{MR1943877}, this means that, for all $t\in\R_+$,
\[
\int_0^t\int_{\R^d}\Bigl(\bigl((K^x_s)^{\top}y\bigr)^2\wedge\bigl|(K^x_s)^{\top}y\bigr|\Bigr)\nu^{S,\FF}(\ud s,\ud y)<\infty
\qquad\text{ $\prob$-a.s. for all }x\in E.
\]
By formula \eqref{expec_init}, this last property implies that, for all $t\in\R_+$,
\be	\label{eq:drift_integr_2}	\begin{aligned}
&\prob\left(\int_0^t\int_{\R^d}\Bigl(\bigl((K^L_s)^{\top}y\bigr)^2\wedge\bigl|(K^L_s)^{\top}y\bigr|\Bigr)\nu^{S,\FF}(\ud s,\ud y)<\infty\right)	\\
&\quad = \int_E\expec\left[q^x_t\ind_{\{\int_0^t\int_{\R^d}(((K^x_s)^{\top}y)^2\wedge|(K^x_s)^{\top}y|)\nu^{S,\FF}(\ud s,\ud y)<\infty\}}\right]\lambda(\ud x)
= \int_E\expec[q^x_t]\lambda(\ud x) = 1.
\end{aligned}	\ee
At this point, making use of \eqref{eq:drift_integr_1}-\eqref{eq:drift_integr_2}, the same arguments given between equations (4.9) and (4.13) in \cite{Jac85} (with $(K^L_t)^{\top}y=:W(t,y)$ and $\nu^{S,\FF}(\ud t,\ud y)=:\nu(\ud t,\ud y)$, in the notation of \cite{Jac85}) allow to deduce that, for all $t\in\R_+$,
\[
\int_0^t\int_{\R^d}\left|(K^L_s)^{\top}yU^L(s,y)\right|\nu^{S,\FF}(\ud s,\ud y)<\infty
\qquad \text{$\prob$-a.s.}
\]
Since the process $A^{(d)}$ admits the representation $A^{(d)}=\int_0^{\cdot}\int_{\R^d}yU^L(s,y)\nu^{S,\FF}(\ud s,\ud y)$ (see \cite[proof of Theorem 2.5]{Jac85}), this proves that $K^L\in L^0(A^{(d)};\prob,\GG)$.\\
(iv):
since $S^d$ is a special semimartingale in $\GG$, the $\bG$-local martingale $\Stilde^{(d)}$ admits the representation $\Stilde^{(d)}=y\ast(\mu^S-\nu^{S,\GG})$ (see \cite[Proposition 3.77]{MR542115}). 
The $\bG$-predictable process $K^L$ belongs to $L_m(\Stilde^{(d)};\prob,\GG)$ if and only if the process $(\sum_{0<s\leq \cdot}((K^L_s)^{\top}\Delta\Stilde^{(d)}_s)^2)^{1/2}$ is $\bG$-locally integrable. 
Since $\Delta\Stilde^{(d)}=\Delta S^d-\Delta A^{(d)}=\Delta S-\Delta A^{(d)}$, it holds that
\begin{align}
\Biggl(\sum_{0<s\leq \cdot}\bigl((K^L_s)^{\top}\Delta\Stilde^{(d)}_s\bigr)^2\Biggr)^{1/2}
&= \Biggl(\sum_{0<s\leq \cdot}\bigl((K^L_s)^{\top}\Delta S_s-(K^L_s)^{\top}\Delta A^{(d)}_s\bigr)^2\Biggr)^{1/2}	\notag\\
&\leq \Biggl(\sum_{0<s\leq \cdot}\bigl((K^L_s)^{\top}\Delta S_s\bigr)^2\Biggr)^{1/2}
+ \Biggl(\sum_{0<s\leq \cdot}\bigl((K^L_s)^{\top}\Delta A^{(d)}_s\bigr)^2\Biggr)^{1/2}	\notag\\
&\leq \Biggl(\sum_{0<s\leq \cdot}\bigl((K^L_s)^{\top}\Delta S_s\bigr)^2\Biggr)^{1/2}
+ \sum_{0<s\leq \cdot}\bigl|(K^L_s)^{\top}\Delta A^{(d)}_s\bigr|.
\label{eq:inequalities}
\end{align}
The last two processes are both $\bG$-locally integrable. Indeed, as argued in step (iii), $(q^L_tm^L_t)_{t\geq0}$ is a special semimartingale in $\bG$, so that the process $(\sum_{0<s\leq\cdot}(\Delta(q^Lm^L)_s)^2)^{1/2}$ is $\bG$-locally integrable. Since $(K^L)^{\top}\Delta S=\Delta(q^Lm^L)$, as shown in step (iii), it follows that the first process appearing in \eqref{eq:inequalities} is $\bG$-locally integrable. The second process appearing in \eqref{eq:inequalities} is also $\bG$-locally integrable since, as shown in step (iii), $K^L\cdot A^{(d)}$ is well-defined as a finite variation process.
This proves that $K^L\in L_m(\Stilde^{(d)};\prob,\GG)$, thus completing the proof.
\end{proof}

\begin{rem}
Note that the stochastic integral $K^L\cdot S$ admits two possible interpretations: the first, as the stochastic integral of the $\bG$-predictable process $K^L$ with respect to $S$; the second, as an $x$-measurable version of the stochastic integral $K^x\cdot S$ evaluated at $x=L$. It is actually part of the result of Proposition \ref{lem:repr_1} that the two interpretations coincide (up to indistinguishability) when viewed in the enlarged filtration $\bG$.
\end{rem}

\begin{rem}	\label{rem:aux_repr}
Proposition \ref{lem:repr_1}, together with \cite[Proposition 3.4]{AFK}, shows that the process $(1/q^L_t)_{t\geq0}$ provides a precise link between $\GG$-martingales and $\FF$-martingales. In our context, the process $(1/q^L_t)_{t\geq0}$ plays a role analogous to the density of the martingale preserving probability measure mentioned in the introduction.
\end{rem}

The following lemma gives a more explicit description of the structure of the process $(1/q^L_t)_{t\geq0}$ and is based on \cite[Lemma 5]{ACJ15}.

\begin{lem}	\label{lem:q_Aks}
Suppose that Assumption \ref{ass:Jacod} holds and that $S=(S_t)_{t\geq0}$ has the strong predictable representation property on $(\Omega,\FF,\prob)$. Then the process $(1/q^L_t)_{t\geq0}$ admits the representation
\be	\label{eq:q_Aks}
\frac{1}{q^L_t} = \frac{1}{q^L_0} - \frac{H^L}{(q^L_{-})^2}\cdot\left(S-\frac{1}{q^L}\cdot [S,q^L]\right),
\ee
where the $\pare{\cE\otimes\cP(\FF)}$-measurable function $(x,\omega,t)\mapsto H^x_t(\omega)$ is as in \eqref{eq:repr_q}.
\end{lem}
\begin{proof}
Recall first that $1/q^L_t$ is well-defined, for all $t\in\R_+$, by \cite[Corollary 1.11]{Jac85}. 
The same arguments used in the proof of Proposition \ref{lem:repr_1} allow to show that the stochastic integral $H^x\cdot S$ appearing in \eqref{eq:repr_q} is well-defined also in the enlarged filtration $\GG$ and admits a version which is measurable in $x$. Moreover, it holds that $H^L\in L(S;\prob,\GG)$ and $H^L\cdot S=(H^x\cdot S)\bigr|_{x=L}=q^L-q^L_0$ $\prob$-a.s.
Similarly as in \cite[Lemma 5]{ACJ15}, an application of It\^o's formula together with the fact that $[q^L,q^L]=\bigl[(q^L)^c,(q^L)^c\bigr]+\sum_{0< u\leq\cdot}(\Delta q^L_u)^2$ then implies that
\[	\ba
\frac{1}{q^L}
&= \frac{1}{q^L_0} - \frac{1}{(q^L_{-})^2}\cdot q^L
+ \frac{1}{(q^L_{-})^3}\cdot\bigl[ (q^L)^c,(q^L)^c\bigr]
+ \sum_{0<u\leq\cdot}\left(\frac{1}{q^L_u}-\frac{1}{q^L_{u-}}+\frac{\Delta q^L_u}{(q^L_{u-})^2}\right)	\\
&= \frac{1}{q^L_0} - \frac{1}{(q^L_{-})^2}\cdot q^L
 + \frac{1}{q^L(q^L_{-})^2}\cdot[q^L,q^L]	\\
 &= \frac{1}{q^L_0} - \frac{1}{(q^L_{-})^2}\cdot\left(q^L-\frac{1}{q^L}\cdot[q^L,q^L]\right)	\\
 &= \frac{1}{q^L_0} - \frac{H^L}{(q^L_{-})^2}\cdot\left(S-\frac{1}{q^L}\cdot[S,q^L]\right),
\ea	\]
where the last equality follows from the associativity of the stochastic integral.
\end{proof}

\subsection{Proofs of the main results stated in Section \ref{sec:main_results}}	\label{proofs_main}

We now give the proofs of the results presented in Section \ref{sec:main_results}.
 In view of Proposition \ref{lem:repr_1} and Lemma \ref{lem:q_Aks}, we are  in a position to complete the proof of Theorem \ref{thm:main}. At this final step, the result of Proposition \ref{prop:Aks} is crucial.

\begin{proof}[Proof of Theorem \ref{thm:main}]
Let $M=(M_t)_{t\geq0}\in\M(\prob,\GG)$. By Proposition \ref{lem:repr_1}, there exists an $\Real^d$-valued $\pare{\cE\otimes\cP(\FF)}$-measurable function $(x,\omega,t)\mapsto K^x_t(\omega)$ such that \eqref{eq:repr_1} holds. 
The integration by parts formula then implies that
\[	\ba
M &=
M_0 + \frac{K^L}{q^L_-}\cdot S + \bigl(q^L_0M_0+(K^L\cdot S)_-\bigr)\cdot\frac{1}{q^L} + K^L\cdot \Bigl[S,\frac{1}{q^L}\Bigr]	\\
&= M_0 + \frac{K^L}{q^L_-}\cdot\left(S-\frac{1}{q^L}\cdot[S,q^L]\right) + \bigl(q^L_0M_0+(K^L\cdot S)_-\bigr)\cdot\frac{1}{q^L},
\ea	\]
where the second equality makes use of the fact that
\[
\Bigl[S,\frac{1}{q^L}\Bigr] = -\frac{1}{q^Lq^L_-}\cdot[S,q^L],
\]
as can be readily verified by an application of It\^o's formula.
Continuing, Lemma \ref{lem:q_Aks} and then Proposition \ref{prop:Aks} imply that
\begin{align}
M &= M_0 + \left(\frac{K^L}{q^L_-}-\bigl(q^L_0M_0+(K^L\cdot S)_-\bigr)\frac{H^L}{(q^L_-)^2}\right)\cdot\left(S-\frac{1}{q^L}\cdot[S,q^L]\right)	\notag\\
&= M_0 + \left(\frac{K^L}{q^L_-}-\bigl(q^L_0M_0+(K^L\cdot S)_-\bigr)\frac{H^L}{(q^L_-)^2}\right)\cdot S^{\GG}	\label{eq:proof_thm}\\
&\quad -  \left(\frac{K^L}{q^L_-}-\bigl(q^L_0M_0+(K^L\cdot S)_-\bigr)\frac{H^L}{(q^L_-)^2}\right)\cdot
\left(\Delta S_{\eta^x}\ind_{\dbraco{\eta^x,\infty}}\right)^{p,\FF}\bigr|_{x=L}.
\notag
\end{align}
Focusing on the last term in the above representation, it holds that
\begin{gather}
\left(\frac{K^L}{q^L_-}-\bigl(q^L_0M_0+(K^L\cdot S)_-\bigr)\frac{H^L}{(q^L_-)^2}\right)\cdot
\left(\Delta S_{\eta^x}\ind_{\dbraco{\eta^x,\infty}}\right)^{p,\FF}\bigr|_{x=L}	\notag\\
= \frac{1}{q^L_-}
\cdot\left(K^x_{\eta^x}\Delta S_{\eta^x}\ind_{\dbraco{\eta^x,\infty}}
-\bigl(q^x_0m^x_0+(K^x\cdot S)_{\eta^x-}\bigr)\frac{H^x_{\eta^x}\Delta S_{\eta^x}}{q^x_{\eta^x-}}\ind_{\dbraco{\eta^x,\infty}}\right)^{p,\FF}\biggr|_{x=L}.
\label{eq:vanish}
\end{gather}
In view of \eqref{eq:prod_repr}, it holds that $\oPP$-a.s. on the set $\{\eta^x<\infty\}$
\[
K^x_{\eta^x}\Delta S_{\eta^x}
= \Delta (q^xm^x)_{\eta^x}
= -q^x_{\eta^x-}m^x_{\eta^x-}
\]
and, recalling representation \eqref{eq:repr_q}, on the set $\{\eta^x<\infty\}$ it holds that
\[
\bigl(q^x_0m^x_0+(K^x\cdot S)_{\eta^x-}\bigr)\frac{H^x_{\eta^x}\Delta S_{\eta^x}}{q^x_{\eta^x-}}
= q^x_{\eta^x-}m^x_{\eta^x-}\frac{\Delta q^x_{\eta^x}}{q^x_{\eta^x-}}
= -q^x_{\eta^x-}m^x_{\eta^x-},
\]
thus showing that \eqref{eq:vanish} vanishes $\prob$-a.s. Together with \eqref{eq:proof_thm} and defining the $\GG$-predictable process
\be	\label{eq:integrand}
\varphi:=\frac{K^L}{q^L_-}-\bigl(q^L_0M_0+(K^L\cdot S)_-\bigr)\frac{H^L}{(q^L_-)^2},
\ee 
this shows that every martingale on $\basisgp$ can be represented in the form $M=M_0+\varphi\cdot S^{\GG}$. Since the general case follows by localization (see e.g. \cite[Lemma 13.2]{MR1219534}), this completes the proof of Theorem \ref{thm:main}.
\end{proof}

\begin{proof}[Proof of Corollary \ref{cor:no_jump}]
Note first that $\prob(\eta^x<\infty)=0$ for $\lambda$-a.e. $x\in E$ implies that the process $\Delta S_{\eta^x}\ind_{\dbraco{\eta^x,\infty}}$ appearing in the decomposition \eqref{eq:dec_Aks} is null\footnote{Since the separability of $L^1(\Omega,\A,\prob)$ is only needed to ensure the existence of an $x$-measurable version of the dual $\FF$-predictable projection of this process (see Remark \ref{rem:sep}), this explains why the separability assumption is not needed in the formulation of Corollary \ref{cor:no_jump}, Proposition \ref{prop:classic_case} and Corollary \ref{cor:no_jump_2}.}, for $\lambda$-a.e. $x\in E$.
By \cite[Corollary 1.11]{Jac85}, the process $(1/q^L_t)_{t\geq0}$ is well-defined.
The $\GG$-local martingale property of the $\Real^{d+1}$-valued process $Y:=(1/q^L,S/q^L)$ follows from \cite[Proposition 3.6]{AFK} (or also  \cite[Proposition 9]{ACJ15}).
Recalling that $[S,\frac{1}{q^L}]=-\frac{1}{q^Lq^L_-}\cdot[S,q^L]$ (see the proof of Theorem \ref{thm:main}) and using the integration by parts formula, it holds that
\begin{align}
S^{\GG} &= S - \frac{1}{q^L}\cdot[S,q^L]
= S + q^L_-\cdot\Bigl[S,\frac{1}{q^L}\Bigr]
= S + q^L_-\cdot\left(\frac{S}{q^L}-S_-\cdot\frac{1}{q^L}-\frac{1}{q^L_-}\cdot S\right)	\notag\\
&= S_0 + q^L_-\cdot \frac{S}{q^L} - (q^L_-S_-)\cdot\frac{1}{q^L}.
\label{eq:no_jump}
\end{align} 
For any $\varphi\in L(S^{\GG};\prob,\GG)$, let $\varphi^n:=\varphi\ind_{\{\|\varphi\|\leq n\}}$, for $n\in\N$. For every $n\in\N$, since $\varphi^n$ is bounded and $q^L_-$ and $S_-$ are both locally bounded, \eqref{eq:no_jump} implies that $\varphi^n\cdot S^{\GG}=\psi^n\cdot Y$, where $(\psi^n_t)_{t\geq0}$ is the $\Real^{d+1}$-valued $\GG$-predictable locally bounded process defined by $\psi^{n,1}:=-q^L_-S^{\top}_-\varphi^n$ and $\psi^{n,i+1}:=q^L_-\varphi^{n,i}$, for all $i=1,\ldots,d$. 
Since $\varphi\in L(S^{\GG};\prob,\GG)$, the stochastic integral $\varphi^n\cdot S^{\GG}$ converges to $\varphi\cdot S^{\GG}$ in Emery's topology. By \cite[Proposition III.6.26]{MR1943877}, this implies that $\psi^n\cdot Y$ also converges in Emery's topology to $\psi\cdot Y$, for some $\psi\in L(Y;\prob,\GG)$, thus showing that 
\[
\bigl\{\varphi\cdot S^{\GG} : \varphi\in L(S^{\GG};\prob,\GG)\bigr\}
\subseteq
\bigl\{\psi\cdot Y : \psi\in L(Y;\prob,\GG)\bigr\}.
\]
The claim then follows from the strong predictable representation property of $S^{\GG}$ on $\basisgp$, which holds by Theorem \ref{thm:main}.
\end{proof}

Let us now turn to the proof of Proposition \ref{prop:classic_case}. Even though Proposition \ref{prop:classic_case} can be obtained as a special case of Corollary \ref{cor:no_jump}, it turns out that it can be very quickly proved by relying on the auxiliary representation result of Proposition \ref{lem:repr_1}, as shown in the following proof.

\begin{proof}[Proof of Proposition \ref{prop:classic_case}]
If $\nu_t\sim\lambda$ holds $\prob$-a.s. for all $t\in[0,T]$, for some $T<\infty$, \cite[Theorem 3.1]{Amen} implies that $(q^L_0/q^L_t)_{t\in[0,T]}\in\M(\prob,\GG)$, so that the probability measure $\ptilde$ is well-defined and  equivalent to $\prob$, since $q^L_T>0$ $\prob$-a.s. Moreover, by \cite[Theorem 3.2]{Amen}, it holds that $\Mloc(\prob,\FF)\subseteq\Mloc(\ptilde,\GG)$, so that $S\in\Mloc(\ptilde,\GG)$. 
Let $N=(N_t)_{t\in[0,T]}\in\M(\ptilde,\GG)$. By Bayes' rule $(q^L_0N_t/q^L_t)_{t\in[0,T]}\in\M(\prob,\GG)$ and, hence, Proposition \ref{lem:repr_1} gives the existence of a $\GG$-predictable process $K^L$ such that 
\[
N_t = N_0 + \frac{1}{q^L_0}(K^L\cdot S)_t = N_0 + \left(\frac{K^L}{q^L_0}\cdot S\right)_t
\qquad\text{ $\prob$-a.s. for all $t\geq0$}.
\]
The general case follows by localization (see \cite[Lemma 13.2]{MR1219534}).
\end{proof}

We conclude with the proof of Corollary \ref{cor:no_jump_2}, which is based on Corollary \ref{cor:no_jump} together with arguments similar to those used in the proof of Proposition \ref{prop:classic_case}.

\begin{proof}[Proof of Corollary \ref{cor:no_jump_2}]
As in Corollary \ref{cor:no_jump}, the fact that $\prob(\eta^x<\infty)=0$ for $\lambda$-a.e. $x\in E$ implies that $(1/q^L_t)_{t\geq0}\in\Mloc(\prob,\GG)$. Let $\{\tau_n\}_{n\in\N}$ be a sequence of $\GG$-stopping times increasing $\prob$-a.s. to infinity such that $(1/q^L_{\tau_n\wedge t})_{t\geq0}$ is a uniformly integrable martingale on $\basisgp$, for all $n\in\N$. Similarly as in the proof of Proposition \ref{prop:classic_case}, define the measure $\ud\ptilde^n:=(q^L_0/q^L_{\tau_n})\ud\prob$ and the stopped filtration $\GG^n=(\G_{\tau_n\wedge t})_{t\geq0}$, for each $n\in\N$. Since $S/q^L\in\Mloc(\prob,\GG)$ (see Corollary \ref{cor:no_jump}), Bayes' rule implies that $(S_{\tau_n\wedge t})_{t\geq0}\in\Mloc(\ptilde^n,\GG^n)$, for all $n\in\N$. 
Let $N=(N_t)_{t\geq0}\in\M(\ptilde^n,\GG^n)$. Bayes' rule implies that $(q^L_0N_t/q^L_{\tau_n\wedge t})_{t\geq0}\in\M(\prob,\GG^n)$ and, by relying on the representation \eqref{eq:repr_1}, the process $N$ admits a representation as a stochastic integral of $S$. Hence, the stopped process $S^{\tau_n}$ has the strong predictable representation property on $(\Omega,\GG^n,\ptilde^n)$, for every $n\in\N$.
Since $S$ is a special semimartingale on $(\Omega,\GG^n,\prob)$ in view of \cite[Theorem 2.5]{Jac85} and $\ptilde^n\sim\prob$, for every $n\in\N$, \cite[Corollary 7.29]{MR542115} implies that $[S,q^L/q^L_0]^{\tau_n}$ is of locally integrable variation on $(\Omega,\GG^n,\ptilde^n)$, for every $n\in\N$. Then, as a consequence of \cite[Theorem 13.12]{MR1219534} (noting that its proof carries over to the $d$-dimensional case), the stopped process $(\bar{S}^{\GG})^{\tau_n}=S^{\tau_n}-\frac{1}{q^L_-}\cdot\langle S,q^x\rangle^{\FF,\tau_n}\bigr|_{x=L}\in\Mloc(\prob,\GG^n)$ has the strong predictable representation property on $(\Omega,\GG^n,\prob)$, for every $n\in\N$. Since the sequence $\{\tau_n\}_{n\in\N}$ increases $\prob$-a.s. to infinity, the claim then follows by \cite[Lemma 13.2]{MR1219534}.
\end{proof}

\appendix

\section{A representation result for a family of $\FF$-martingales}	\label{sec:alt}

As explained in Section \ref{outline}, we have introduced the product space $(\widehat{\Omega},\oFF,\oPP)$ in order to establish the martingale representation result of Proposition \ref{prop:mr}, which provides a stochastic integral representation holding simultaneously for a family $\{(m^x_t)_{t\geq0} : x\in E\}$ of $\FF$-martingales depending on $x$ in a measurable way, thereby ensuring good measurability properties of the integrands appearing in the representation. 

As mentioned in Remark \ref{rem:Jacod}, one can directly prove the existence of a simultaneous martingale representation result on the original space $\basisp$. This is the content of the following proposition, which allows to avoid the development of Subsection \ref{subsec2}, at the expense of relying on more sophisticated tools.
We denote by $\mu^S(\omega;dt,dy)$ the jump measure of the $\Real^d$-valued $\FF$-local martingale $S$ and by $\nu^{S,\FF}(\omega;dt,dy)$ the corresponding compensating measure in the filtration $\FF$. If $\Omega\times\Real_+\times\Real^d\ni(\omega,t,y)\mapsto W(\omega,t,y)$ is a $(\cP(\FF)\otimes\mathcal{B}_{\Real^d})$-measurable function, we denote by $W\ast(\mu^S-\nu^{S,\FF})$ the stochastic integral (when it exists) with respect to the random measure $\mu^S-\nu^{S,\FF}$, in the sense of \cite[Definition II.1.27]{MR1943877}. We denote by $\mathcal{G}_{\text{loc}}(\mu^S;\FF)$ the set of all  $\pare{\cP(\FF)\otimes\mathcal{B}_{\Real^d}}$-measurable functions $(\omega,t,y)\mapsto W(\omega,t,y)$ such that the stochastic integral $W\ast(\mu^S-\nu^{S,\FF})$ exists. We also denote by $S^c$ and $S^d$ the continuous and purely discontinuous, respectively, $\FF$-local martingale parts of $S$ (see \cite[Theorem I.4.18]{MR1943877}).

\begin{prop}	\label{prop:alt}
Suppose that $S=(S_t)_{t\geq0}$ has the strong predictable representation property on $\basisp$. Let $E\times\Omega\times\Real_+\ni(x,\omega,t)\mapsto m^x_t(\omega)$ be a $\pare{\cE\otimes\cO(\FF)}$-measurable function such that $(m^x_t)_{t\geq0}\in\M(\prob,\FF)$, for all $x\in E$. Then there exists a $\pare{\cE\otimes\cP(\FF)}$-measurable function $E\times\Omega\times\Real_+\ni(x,\omega,t)\mapsto\theta^x_t(\omega)\in\Real^d$ satisfying $\theta^x\in L_m(S;\prob,\FF)$, for all $x\in E$, such that $m^x_t(\omega)=m^x_0(\omega)+(\theta^x\cdot S)_t(\omega)$ holds $\prob$-a.s. for all $t\in\R_+$ and all $x\in E$.
\end{prop}
\begin{proof}
Since the $\FF$-local martingale $S$ has the strong predictable representation property on $\basisp$ and the strong predictable representation property implies the weak predictable representation property (see e.g. \cite[Theorem 13.14]{MR1219534}), it holds that
\be	\label{eq:jac_1}
m^x = m^x_0 + h^x\cdot S^c + W^x\ast(\mu^S-\nu^{S,\FF}),
\qquad\text{ for every }x\in E,
\ee
where $(h^x_t)_{t\geq0}\in L_m(S^c;\prob,\FF)$ and $W^x\in\mathcal{G}_{\text{loc}}(\mu^S;\FF)$, for every $x\in E$.
More precisely, the process $(h^x_t)_{t\geq0}$ can be chosen as any $\FF$-predictable process such that 
\be	\label{eq:qv}
[m^x,S^c] =  [(m^x)^c,S^c] = h^x\cdot [S^c,S^c], 
\qquad \text{ for every $x\in E$.}
\ee 
Since $[m^x,S^c]$ is measurable in $x$ (see \cite[Proposition 2]{SY}), one can find a $(\cE\otimes\cP(\FF))$-measurable function $(x,\omega,t)\mapsto h^x_t(\omega)\in\Real^d$ which satisfies \eqref{eq:qv}. Moreover, by following the construction in part b of \cite[proof of Proposition 3.14]{Jac85}, one can also find a $(\cE\otimes\cP(\FF)\otimes\mathcal{B}_{\Real^d})$-measurable  function $(x,\omega,t,y)\mapsto W^x(\omega,t,y)$ such that $\Delta m^x_t(\omega)=W^x(\omega,t,\Delta S_t(\omega))\ind_{\{\Delta S_t(\omega)\neq0\}}$ up to an evanescent set, for every $x\in E$.
Since $S$ has the strong predictable representation property on $\basisp$, \cite[Theorem 4.80]{MR542115} implies that there exists a family $\{(\alpha_i(t,\cdot))_{t\geq0} : i=1,\ldots,d+1\}$ of $\Real^d$-valued $\FF$-predictable processes such that the set $B(t,\omega):=\{\alpha_i(t,\omega) : i=1,\ldots,d+1\}$ is composed of linearly independent points and satisfies $\Delta S_t(\omega)\in B(t,\omega)$ (up to an evanescent set).
For each $x\in E$, define then the $\R^{d+1}$-valued $\FF$-predictable process $(V^x_t)_{t\geq0}$ by
\[
V^{x,i}_t(\omega) := W^x\bigl(\omega,t,\alpha_i(t,\omega)\bigr),
\qquad\text{ for }i=1,\ldots,d+1\text{ and }t\in\R_+.
\]
Observe that the process $(V^x_t)_{t\geq0}$ is measurable in $x$.
Following the construction in the proof of the implication (b)$\Rightarrow$(a) of \cite[Theorem 4.80]{MR542115}, the purely discontinuous $\FF$-local martingale $W^x\ast(\mu^S-\nu^{S,\FF})$ can be represented as a stochastic integral of the form $H^x\cdot S^d$ if there exists an $\FF$-predictable solution $(H^x_t)_{t\geq0}$ to the linear system
\be	\label{eq:system}
V^{x,i} - Z^x = \sum_{j=1}^dH^{x,j}\alpha_i^j,
\qquad\text{ for }i=1,\ldots,d+1,
\ee
where, following the notation of \cite{MR542115}, $(Z^x_t)_{t\geq0}$ is an $\FF$-predictable process which depends on $x$ in a measurable way. As argued in \cite{MR542115}, system \eqref{eq:system} admits a unique solution $(H^x_t)_{t\geq0}$ and the map $(x,\omega,t)\mapsto H^x_t(\omega)$ is $(\cE\otimes\cP(\FF))$-measurable, since all the processes appearing in \eqref{eq:system} are $(\cE\otimes\cP(\FF))$-measurable.
We have thus established the representation 
\be	\label{eq:sim_repr}
m^x=m^x_0+h^x\cdot S^c+H^x\cdot S^d,
\qquad\text{ for all }x\in E,
\ee 
where both integrands depend in a measurable way on $x$. 
Finally, since $S$ has the strong predictable representation property on $\basisp$, \cite[Theorem 4.82]{MR542115} together with \cite[Theorem 4.73]{MR542115} implies that $S^{c,i}\in\{\varphi\cdot S : \varphi\in L_m(S;\prob,\FF)\}$ and $S^{d,i}\in\{\varphi\cdot S : \varphi\in L_m(S;\prob,\FF)\}$, for all $i=1,\ldots,d$. Hence, there exist two $\Real^{d\times d}$-valued $\FF$-predictable processes $(\psi^c_t)_{t\geq0}$ and $(\psi^d_t)_{t\geq0}$ such that the row vectors $(\psi^{c,i\cdot}_t)_{t\geq0}$ and $(\psi^{d,i\cdot}_t)_{t\geq0}$ belong to $L_m(S;\prob,\FF)$, for each $i=1,\ldots,d$, and $S^c=\psi\cdot S$ and $S^d=\psi^d\cdot S$. By the associativity of the local martingale stochastic integral, the claim then follows by letting $(x,\omega,t)\mapsto\theta^x_t(\omega):=\psi^c_t(\omega)^{\top}h^x_t(\omega)+\psi^d_t(\omega)^{\top}H^x_t(\omega)$, which is $(\cE\otimes\cP(\FF))$-measurable.
\end{proof}

\begin{rem}
In the above proof, a key ingredient is represented by the family of $\Real^d$-valued processes $\{(\alpha_i(t,\cdot)_{t\geq0} : i=1,\ldots,d+1)\}$ such that $\Delta S_t\in\{\alpha_i(t,\cdot) : i=1,\ldots,d+1\}$. In the recent paper \cite{Song2015}, this property has been called \emph{finite $\FF$-predictable constraint} condition. In the proof of Proposition \ref{prop:alt}, the construction based on \cite[Theorem 4.80]{MR542115} can be replaced by a direct application of \cite[Theorem 3.5]{Song2015}, using the fact that each $\FF$-local martingale $X_k$ (in the notation of \cite{Song2015}) can be written as a stochastic integral of $S$ as long as $S$ has the strong predictable representation property on $\basisp$.
\end{rem}

\bibliographystyle{alpha}
\bibliography{init_filtration_biblio2}

\begin{thebibliography}{CCFM17}

\bibitem[ABS03]{ABS}
J.~Amendinger, D.~Becherer, and M.~Schweizer.
\newblock A monetary value for initial information in portfolio optimization.
\newblock {\em Finance Stoch.}, 7:29--46, 2003.

\bibitem[ACJ15]{ACJ15}
A.~Aksamit, T.~Choulli, and M.~Jeanblanc.
\newblock On an optional semimartingale decomposition and the existence of a
  deflator in an enlarged filtration.
\newblock In C.~Donati-Martin, A.~Lejay, and A.~Rouault, editors, {\em In
  Memoriam Marc Yor - {S}\'eminaire de Probabilit\'es {XLVII}}, volume 2137 of
  {\em Lecture Notes in Mathematics}, pages 187--218. Springer, 2015.

\bibitem[AFK16]{AFK}
B.~Acciaio, C.~Fontana, and C.~Kardaras.
\newblock Arbitrage of the first kind and filtration enlargements in
  semimartingale financial models.
\newblock {\em Stoch. Proc. Appl.}, 126(6):1761--1784, 2016.

\bibitem[AIS98]{AIS98}
J.~Amendinger, P.~Imkeller, and M.~Schweizer.
\newblock Additional logarithmic utility of an insider.
\newblock {\em Stoch. Proc. Appl.}, 75(2):263--286, 1998.

\bibitem[Ame00]{Amen}
J.~Amendinger.
\newblock Martingale representation theorems for initially enlarged
  filtrations.
\newblock {\em Stoch. Proc. Appl.}, 89(1):101--116, 2000.

\bibitem[Bau03]{Bau03}
F.~Baudoin.
\newblock Modeling anticipations in financial markets.
\newblock In R.~Carmona et~al., editor, {\em Paris-Princeton Lectures in
  Mathematical Finance 2002}, pages 43--94. Springer, Berlin Heidelberg, 2003.

\bibitem[Cam05]{Campi05}
L.~Campi.
\newblock Some results on quadratic hedging with insider trading.
\newblock {\em Stoch. Stoch. Rep.}, 77(4):327--348, 2005.

\bibitem[CCFM17]{CCFM}
H.N. Chau, A.~Cosso, C.~Fontana, and O.~Mostovyi.
\newblock Optimal investment with intermediate consumption under no unbounded
  profit with bounded risk.
\newblock {\em J. Appl. Probab.}, forthcoming, 2017.

\bibitem[CJZ13]{CJZ}
G.~Callegaro, M.~Jeanblanc, and B.~Zargari.
\newblock Carthaginian enlargement of filtrations.
\newblock {\em ESAIM: Probability and Statistics}, 17:550--566, 2013.

\bibitem[CT15]{CT15}
H.N. Chau and P.~Tankov.
\newblock Market models with optimal arbitrage.
\newblock {\em SIAM J. Financial Math.}, 6:66--85, 2015.

\bibitem[DMM92]{DMM}
C.~Dellacherie, B.~Maisonneuve, and P.-A. Meyer.
\newblock {\em Probabilit\'es et Potentiel. {C}hapitres de {XVII} \`a {XXIV}.
  Processus de {M}arkov (fin), Compl\'ements de Calcul Stochastique.}
\newblock Hermann, Paris, 1992.

\bibitem[DS94]{MR1304434}
F.~Delbaen and W.~Schachermayer.
\newblock A general version of the fundamental theorem of asset pricing.
\newblock {\em Math. Ann.}, 300(3):463--520, 1994.

\bibitem[EI15]{EsmImk}
N.~Esmaeeli and P.~Imkeller.
\newblock American options with asymmetric information and reflected {BSDE}.
\newblock {\em Bernoulli}, forthcoming, 2015.

\bibitem[EKJJ10]{EKJJ}
N.~El~Karoui, M.~Jeanblanc, and Y.~Jiao.
\newblock What happens after a default: The conditional density approach.
\newblock {\em Stoch. Proc. Appl.}, 120(7):1011--1032, 2010.

\bibitem[EL05]{EL05}
A.~Eyraud-Loisel.
\newblock Backward stochastic differential equations with enlarged filtrations:
  option hedging of an insider trader in a financial market with jumps.
\newblock {\em Stoch. Proc. Appl.}, 115(11):1745--1763, 2005.

\bibitem[FI93]{FI93}
H.~F{\"o}llmer and P.~Imkeller.
\newblock Anticipation cancelled by a {G}irsanov transformation: a paradox on
  {W}iener space.
\newblock {\em Annales de l'{I.H.P.}, section B}, 29(4):569--586, 1993.

\bibitem[GP98]{GP98}
A.~Grorud and M.~Pontier.
\newblock Insider trading in a continuous time market model.
\newblock {\em Int. J. Theor. Appl. Finance}, 1(3):331--347, 1998.

\bibitem[GP01]{MR1831271}
A.~Grorud and M.~Pontier.
\newblock Asymmetrical information and incomplete markets.
\newblock {\em Int. J. Theor. Appl. Finance}, 4(2):285--302, 2001.

\bibitem[GVV06]{GVV06}
D.~Gasbarra, E.~Valkeila, and L.~Vostrikova.
\newblock Enlargement of filtration and additional information in pricing
  models: {B}ayesian approach.
\newblock In Y.~Kabanov, R.~Liptser, and J.~Stoyanov, editors, {\em From
  Stochastic Calculus to Mathematical Finance}, pages 257--285. Springer,
  Berlin, 2006.

\bibitem[HWY92]{MR1219534}
S.W. He, J.G. Wang, and J.A. Yan.
\newblock {\em Semimartingale Theory and Stochastic Calculus}.
\newblock Kexue Chubanshe (Science Press), Beijing, 1992.

\bibitem[Jac79]{MR542115}
J.~Jacod.
\newblock {\em Calcul Stochastique et Probl\`emes de Martingales}, volume 714
  of {\em Lecture Notes in Mathematics}.
\newblock Springer, Berlin, 1979.

\bibitem[Jac85]{Jac85}
J.~Jacod.
\newblock Grossissement initial, hypoth\`ese ({H}'), et th\'eor\`eme de
  {G}irsanov.
\newblock In T.~Jeulin and M.~Yor, editors, {\em Grossissements de Filtrations:
  Exemples et Applications}, volume 1118 of {\em Lecture Notes in Mathematics},
  pages 15--35. Springer, Berlin - Heidelberg, 1985.

\bibitem[Jeu80]{MR604176}
T.~Jeulin.
\newblock {\em Semi-martingales et Grossissement d'une Filtration}, volume 833
  of {\em Lecture Notes in Mathematics}.
\newblock Springer, Berlin, 1980.

\bibitem[JLC09]{JLC09}
M.~Jeanblanc and Y.~Le~Cam.
\newblock Progressive enlargement of filtrations with initial times.
\newblock {\em Stoch. Proc. Appl.}, 119(8):2523--2543, 2009.

\bibitem[JLC10]{JLC2010}
M.~Jeanblanc and Y.~Le~Cam.
\newblock Immersion property and credit risk modelling.
\newblock In F.~Delbaen, M.~R\'asony, and C.~Stricker, editors, {\em Optimality
  and Risk: Modern Trends in Mathematical Finance}, pages 99--132. Springer,
  2010.

\bibitem[JS03]{MR1943877}
J.~Jacod and A.N. Shiryaev.
\newblock {\em Limit Theorems for Stochastic Processes}, volume 288 of {\em
  Grundlehren der Mathematischen Wissenschaften}.
\newblock Springer, Berlin, second edition, 2003.

\bibitem[JS15]{JS15}
M.~Jeanblanc and S.~Song.
\newblock Martingale representation property in progressively enlarged
  filtrations.
\newblock {\em Stoch. Proc. Appl.}, 125(11):4242--4271, 2015.

\bibitem[JYC09]{JYC09}
M.~Jeanblanc, M.~Yor, and M.~Chesney.
\newblock {\em Mathematical Methods for Financial Markets}.
\newblock Springer Finance. Springer, London, 2009.

\bibitem[KK07]{MR2335830}
I.~Karatzas and C.~Kardaras.
\newblock The num\'eraire portfolio in semimartingale financial models.
\newblock {\em Finance Stoch.}, 11(4):447--493, 2007.

\bibitem[KLP13]{KLP}
Y.~Kchia, M.~Larsson, and P.~Protter.
\newblock Linking progressive and initial filtration expansions.
\newblock In F.~Viens, J.~Feng, Y.~Hu, and E.~Nualart, editors, {\em Malliavin
  Calculus and Stochastic Analysis: A Festschrift in Honor of David Nualart},
  volume~34 of {\em Springer Proceedings in Mathematics \& Statistics}, pages
  469--487. Springer, 2013.

\bibitem[Pro04]{MR1037262}
P.~Protter.
\newblock {\em Stochastic Integration and Differential Equations}.
\newblock Applications of Mathematics. Springer, Berlin, 2.1 edition, 2004.

\bibitem[Son87]{Song_thesis}
S.~Song.
\newblock {\em Grossissement d'une filtration et probl\`emes connexes}.
\newblock PhD thesis, Universit\'e Paris VI, 1987.

\bibitem[Son14]{Song14}
S.~Song.
\newblock Optional splitting formula in a progressively enlarged filtration.
\newblock {\em ESAIM: Probability and Statistics}, 18:829--853, 2014.

\bibitem[Son15a]{Song_local}
S.~Song.
\newblock An introduction of the enlargement of filtration.
\newblock Preprint available at \url{https://arxiv.org/abs/1510.05212}, 2015.

\bibitem[Son15b]{Song2015}
S.~Song.
\newblock Martingale representation processes and applications in the market
  viability with information flow expansion.
\newblock Preprint available at \url{http://arxiv.org/abs/1505.00560}, 2015.

\bibitem[SY78]{SY}
C.~Stricker and M.~Yor.
\newblock Calcul stochastique d\'ependant d'un param\`etre.
\newblock {\em Z. Wahrsch. Verw. Gebiete}, 45(2):109--133, 1978.

\end{thebibliography}
\end{document}